\documentclass{article}

\usepackage[english]{babel}

\usepackage[letterpaper,top=2cm,bottom=2cm,left=3cm,right=3cm,marginparwidth=1.75cm]{geometry}
\usepackage[toc,page]{appendix}

\usepackage[algo2e,linesnumbered,vlined,ruled]{algorithm2e}

\usepackage{amsmath,amssymb,amsthm,color}
\usepackage{graphicx}
\usepackage{subcaption}
\usepackage{mdframed}
\usepackage{caption}
\usepackage{overpic}
\usepackage{titlesec}
\usepackage{tikz}
\usepackage[colorlinks=true, allcolors=blue]{hyperref}

\newtheorem{theorem}{Theorem}[section]
\newtheorem{lemma}[theorem]{Lemma}
\newtheorem{remark}[theorem]{Remark}

\newtheorem{proposition}[theorem]{Proposition}
\newtheorem{example}[theorem]{Example}
\newtheorem{definition}[theorem]{Definition}
\newtheorem{assumption}[theorem]{Assumption}
\newtheorem*{remark*}{Remark}

\usepackage{algorithm}
\usepackage{algpseudocode}
\usepackage{verbatim}
\usepackage{xcolor}
\usepackage{mathrsfs}

\newcommand{\cX}{\mathcal{X}}
\newcommand{\cY}{\mathcal{Y}} 
\newcommand{\cZ}{\mathcal{Z}} 

\newcommand{\cA}{\mathcal{A}}
\newcommand{\cB}{\mathcal{B}}
\newcommand{\cG}{\mathcal{G}}

\newcommand{\cF}{\mathcal{F}}

\newcommand{\cO}{\mathcal{O}}

\newcommand{\bL}{\mathbb{L}}

\newcommand{\cT}{\mathcal{T}}

\newcommand{\cP}{\mathcal{P}}

\newcommand{\ctdb}{\mathcal{T}^\delta_{\beta_t}}
\newcommand{\ctdbi}{\mathcal{T}^\delta_{\beta_i^t}}

\newcommand{\git}{g_{i}^t}
\newcommand{\gitm}{g_{i}^{t-1}}
\newcommand{\gite}{g_{i}^{t+1}}
\newcommand{\bGt}{\mathbf{G}^t}
\newcommand{\bGte}{\mathbf{G}^{t+1}}

\newcommand{\bxt}{\mathbf{x}^t}
\newcommand{\bvt}{\mathbf{v}^t}
\newcommand{\bvte}{\mathbf{v}^{t+1}}
\newcommand{\bxte}{\mathbf{x}^{t+1}}
\newcommand{\byt}{\mathbf{y}^t}
\newcommand{\byte}{\mathbf{y}^{t+1}}

\newcommand{\xit}{x_{i}^{t}}
\newcommand{\xitm}{x_{i}^{t-1}}
\newcommand{\xite}{x_{i}^{t+1}}

\newcommand{\yit}{y_{i}^{t}}

\newcommand{\bbxt}{\bar{\mathbf{x}}^t}
\newcommand{\bbxte}{\bar{\mathbf{x}}^{t+1}}

\newcommand{\bbyt}{\bar{\mathbf{y}}^t}
\newcommand{\bbyte}{\bar{\mathbf{y}}^{t+1}}

\newcommand{\bbvte}{\bar{\mathbf{v}}^{t+1}} 
\newcommand{\bbvt}{\bar{\mathbf{v}}^{t}}

\newcommand{\RR}{\mathbb{R}}
\newcommand{\EE}{\mathbb{E}}

\newcommand{\st}{\mathrm{s.t.}}

\newcommand{\bx}{\mathbf{x}}

\newcommand{\bW}{\mathbf{W}}

\newcommand{\bone}{\mathbf{1}}
\newcommand{\lnorm}{\left\|}
\newcommand{\rnorm}{\right\|}

\newcommand{\fij}{f_{i,j}}

\newcommand\keywords[1]{\textbf{Keywords}: #1}

\title{Clipped Stochastic Gradient Tracking For\\  Locally Smooth Functions}
\author{Leilei Mei\thanks{National University of Singapore, Singapore. Email: e0974170@u.nus.edu},   
Junyu Zhang\thanks{National University of Singapore, Singapore. Email: junyuz@nus.edu.sg}}

\date{}

\begin{document}

\maketitle

\begin{abstract}
Most stochastic gradient tracking (GT) methods adopt pre-scheduled stepsize rules, while a few recent works studied adaptive stepsizes that attempt to respond to the problem's local landscape. These methods are typically built upon the problem's global smoothness constant in both analysis and implementation, even for the adaptive ones. On the one hand, for many problems the local smoothness constant may vary drastically across the domain, and sometimes even unbounded, using the global upper bound of the local constants is too conservative. On the other hand, drastic stepsize changes can cause difficulties in the analysis of convergence and consensus of distributed algorithms, making the direct use of local smoothness constants risky and theoretically challenging. In this paper, we propose a \emph{Relative Uniform Continuity} (RUC) regularity condition for the local smoothness constant as a function of sets. The RUC condition covers most common growth functions for local smoothness constant, ranging from constant and logarithmic to polynomial and even exponential. For RUC-regular distributed optimization problems with finite-sum structure, we derive a clipped gradient tracking method with staggered variance reduction, which only relies on the local smoothness of objective functions, and an  $\cO(\sum_in_i^{1.5}+n_i^{0.5}\epsilon^{-1})$ complexity has been established for our algorithm. 
\end{abstract}


\section{Introduction}
\label{section:Intro}
Distributed optimization is a key of large network systems in many signal processing, control, and machine learning applications \cite{nedic2017achieving,olfati2007consensus,chen2012diffusion,ling2012decentralized}, where people aim to solve the following formulation: 
\begin{eqnarray}
    \label{prob:consensus}
    \min_{x_1,\cdots,x_m}\,\, \frac{1}{m}\sum_{i=1}^{m} f_i(x_i)\qquad \st \qquad x_i = x_j, \forall\, i,j\in[m].
\end{eqnarray}
Efficient classic algorithms for this problem include the distributed gradient descent (DGD)   \cite{nedic2009distributed,nedic10distributed,yuan2016convergence,zeng2018nonconvex}, gradient tracking (GT)  \cite{nedic2017achieving,pu2021distributed,koloskova2021improved}, Push-Pull \cite{nedic2009distributed,li2019decentralized,shi2015extra,qu2017harnessing}, the primal-dual (PD) methods \cite{pesquet2014class,zhang2017distributed}, as well as their numerous  variants \cite{shi2015extra,hong2017prox,tang2018d,lu2019gnsd,sun2019distributed,xin2020general,pu2021distributed,sun2022distributed}, to name a few.  In this paper, we are specifically interested in the stochastic distributed algorithms for the finite-sum setting: 
\begin{equation}
    \label{prob:main}
    \min_{x\in\RR^d} \frac{1}{m}\sum_{i=1}^m f_i(x)\qquad\mbox{with}\qquad f_i(x) = \frac{1}{n_i}\sum_{j=1}^{n_i} \fij(x), \quad i\in[m], 
\end{equation}
where $n_i$ stands for the local data size of agent $i$, $[m]:=\{1,\cdots,m\},$ and the local objective function $f_i$'s are differentiable but possibly nonconvex.
Under the finite-sum setting, the predominant results are the distributed stochastic variance reduced first-order methods, including \cite[etc.]{sun2020improving,xin2020general,xin2020variance,xin2020near,xin2021fast,jiang2022distributed}, most of which are based on gradient tracking due to its ability to simultaneously obtain consensus in iterates and gradient estimations. By adopting variance reduction technique, they behave significantly better than vanilla stochastic approximation variants of GT and DGD, such as \cite{tang2018d,lu2019gnsd,pu2021distributed}.

In this paper, we are particularly interested in the problem setting where, instead of assuming each $f_i$ to be globally $L$-smooth, we only require the knowledge of a set function $\bL(\cX)$ that returns an upper bound of $f_i$'s local smoothness constant in some set $\cX$, which does not exclude the extreme case $\bL(\mathbb{R}^d)=+\infty$ that corresponds to problems that are not globally Lipschitz-smooth.

Our main motivation for considering local smoothness is that for many optimization problems, the local smoothness constants can vary drastically across the problem domain. One typical example is the so-called $\alpha$-generalized smoothness condition, which assumes the training loss $f(x)$ to satisfy $\|\nabla^2 f(x)\| \leq L_0 + L_1\|\nabla f(x)\|^\alpha.$ It holds with $\alpha = 1$ for deep models in natural language processing (NLP), computer vision (CV) tasks \cite{zhang2019gradient}, and distributed robust optimization (DRO) tasks \cite{jin2021non}. It also holds with $\alpha = 2/3$ for phase retrieval problem \cite{chen2023generalized}, and it is shown in \cite{li2023convex} that the self-concordant functions \cite{nesterov2013introductory} satisfy this condition with $\alpha=2$. Note that for large problems, the initial gradient size is often large if no good guess of near-stationary initialization is provided, e.g., in our experiments, the initial gradient can be as large as $10^5$. This brings large variation to local smoothness constants during problem solving. In addition, for the hard problem instances considered under the relative smooth optimization topic \cite{bolte2018first,lu2018relatively}, where $\|\nabla^2f(x)\|\leq L\|\nabla^2h(x)\|$ for some general strictly convex but non-Lipschitz-smooth distance generating function $h$, the local smoothness constants can explode when the iterations approach the boundary $\partial\mathrm{dom}(h)$. In these problems, the global smoothness constant either does not exist or 
can be several orders greater than the local constants. According to the above discussion, it is crucial for an algorithm to adapt to a problem's local smoothness properties if the problem has drastically varying local smoothness constants. 

Note that the state-of-the-art variance reduced algorithms \cite[etc.]{sun2020improving,xin2020general,xin2020near,xin2020variance,xin2021fast,jiang2022distributed} in the existing results all work under the classic setting where a global smoothness constant is assumed and used throughout the design and analysis of the algorithm.
Therefore, a straightforward idea is  replacing the global smoothness constant by its local counterparts in each epoch of the variance reduced methods. Unfortunately, such a direct modification may not necessarily lead to a convergent method, which is mainly due to the technical intricacies of GT-based variance reduction methods in balancing consensus and convergence, as well as designing a descending Lyapunov functions that governs the algorithm's convergence behavior. 

To resolve this gap, we introduce a new \emph{Relative Uniform Continuous} (RUC) regularity condition for the local
smoothness constant. Under this condition, for any $\epsilon>0$, there exists a $\delta>0$ so that the \emph{relative difference} between the local smoothness constants satisfies
$$\left|1-\max\left\{\frac{\mathbb{L}(\cX)}{\mathbb{L}(\cY)},\frac{\mathbb{L}(\cY)}{\mathbb{L}(\cX)}\right\}\right| < \epsilon \qquad \forall d_\mathrm{H}(\cX,\cY)\leq \delta,$$ 
where $d_\mathrm{H}(\cdot,\cdot)$ stands for the Hausdorff distance of the two sets. We should notice that the above relative difference bound is much milder than requiring the absolute value difference $|\mathbb{L}(\cX)-\mathbb{L}(\cY)|$ to be small in our interested problem setting, where both $\mathbb{L}(\cX)$ and $\mathbb{L}(\cY)$ are large numbers.  

In this paper, we show that this condition covers a lot of common growth
functions $\mathbb{L}(\cdot)$ for local smoothness constants, ranging from constant and logarithmic to polynomial and even exponential. In particular, the $\alpha$-generalized smoothness condition with $\alpha\in(0,1]$, and the relative smoothness condition with appropriate kernels can also be covered. Moreover, we show that this regularity condition is very versatile and robust and it remains closed under various operations on the function $\mathbb{L}(\cdot)$, such as affine composition, scaling, linear combination, and point-wise maximum or minimum, making it convenient to construct or check new growth function for local smoothness constant by manipulating the known RUC-regular functions.

If the objective function of problem \eqref{prob:main} is RUC-regular, then it is implied that the local smoothness constants do not change so fast in a \emph{relative} sense, even though they may still vary drastically in the \emph{absolute} sense. To utilize this property, we consider a clipped gradient tracking technique that clips the update when it is too aggressive while preserving the update if it is in a controllable range. Based on this technique, we design a staggered variance reduced algorithms with clipped gradient tracking to solve problem \eqref{prob:main}, which allows a staggered, or asynchronous, computation of the check points, as well as the restart of epochs, among the distributed agents. For this method, a tight complexity has been established for our method, which can reduce to the classic $O(\sqrt{n}/\epsilon)$ complexity for finding an $\epsilon$-stationary point in the classic setting where the problem is globally Lipschitz-smooth. 
 \vspace{0.3cm}

\noindent\textbf{\large Related literature.}$\,\,$ There is a rich bulk of literature in distributed optimization. Here, we only review the most related works that focus on the stochastic gradient tracking methods, the variance reduction techniques, the gradient clipping, as well as the locally smooth optimization results. 

The gradient tracking method was originally proposed in \cite{nedic2017achieving} for convex problems, and has been extended to the nonconvex and stochastic setting in \cite{zhang2019decentralized,lu2019gnsd,koloskova2021improved,pu2021distributed}, and they significantly outperform their DGD counterparts, whose convergence requires additional regularity condition such as uniformly bounded difference between global and local gradients \cite{lian2017can,assran2019stochastic}. Moreover, for finite-sum problems like \eqref{prob:main}, the ability of GT to track the average gradient facilitate its combination with the stochastic variance reduction technique \cite{johnson2013accelerating,defazio2014saga,pham2020proxsarah,cutkosky2019momentum}, leading to improved complexity results for both convex and nonconvex settings, see e.g. \cite{li2020communication,sun2020improving,xin2020general,xin2020variance,xin2021fast,jiang2022distributed}. For our problem setting where the local smoothness constants can vary drastically or even unbounded, directly combining GT and variance reduction may not work, and we need to incorporate a gradient clipping technique that is typically used in  neural networks training to prevent the issue of exploding gradients by constraining the gradient to some specific threshold, and has been theoretically analyzed under the generalized smoothness condition \cite[etc.]{zhang2019gradient,koloskova2023revisiting,li2023convex}, where $\epsilon$-small clipping threshold is often adopted, leading to slow conservative updates. The gradient clipping technique has also been applied to the policy gradient analysis in reinforcement learning to prevent the exploding importance weights \cite{zhang2021convergence}. In this work, the gradient clipping will function in a different way in order to improve consensus and to facilitate the use of the RUC-regularity condition, and we only require a constant level clipping threshold. Finally, in terms of the algorithms that exploit the local smoothness, \cite{lu2023accelerated} and \cite{lu2024primal} consider the strongly convex optimization and strongly monotone variational inequality, respectively, and \cite{zhang2024first,hong2020diverge} consider the unconstrained nonconvex optimization problem. Besides the above case-by-case literature, a large class results exists under the topic of relative-smoothness \cite{lu2018relatively} or smooth-adaptable \cite{bauschke2017descent} optimization problems, where a carefully designed kernel function $h$ and the induced Bregman divergence are used to automatically exploit the local smoothness via mirror descent updates, see e.g. \cite{bolte2018first,teboulle2018simplified}. However, as the nonlinear mirror maps $\nabla h$ and $\nabla h^*$ can cause consensus difficulties, we will not adopt this approach.\vspace{0.25cm}

\noindent\textbf{\large Organization.}$\,\,$
The remainder of this paper is organized as follows. In Section~\ref{section:RUC-regularity}, we introduce the concept of relative uniform continuous regularity and its relevance to our problem setting. Section~\ref{section:CGTVR} presents the proposed clipped gradient tracking algorithm with synchronous and staggered variance reduction. Section~\ref{section:cvg-analysis} is devoted to the convergence and complexity analysis of the proposed \texttt{CGTVR-stag} method. Finally, in Section~\ref{section:numerical}, we provide extensive numerical experiments that demonstrate the practical performance of our approach. 

\vspace{0.25cm}

\noindent\textbf{\large Notation.}$\,\,$ If not specifically mentioned, $\|\cdot\|$ will always be defaulted as $\ell_2$ norm for vectors and spectral norm for matrices without causing confusion. For any matrix $A\in\RR^{m\times d}$, we use $\|A\|_F$ to denote the Frobenius norm of $A$. Then it holds that $\|AB\|_F\leq \|A\|\|B\|_F$ for any two matrices $A,B$ of compatible sizes. We use $\|A\|_{2,\infty}$ to denote its $\ell_2$-$\ell_\infty$ norm defined by $\|A\|_{2,\infty} = \max_{i}\|A_{i,:}\|$, where $A_{i,:}$ stands for the $i$-th row of matrix $A$. For any vector $v$, the $\ell_2$-$\ell_\infty$ norm satisfies $\|A^\top v\|\leq \|A\|_{2,\infty}\cdot\|v\|_1$, where $\|v\|_1:=\sum_i |v_i|$ denotes the $\ell_1$ norm of $v$. We use $B(x, r)$ to denote a ball centered at $x$ with radius $r$. We use $\mathbf{1}$ to denote all-one vector. We use $[x,y]$ to denote the line segment between two points  $x,y\in\RR^d.$ We denote $[m] := \{1, 2, \cdots, m\}.$

\section{Relative uniform continuous regularity condition}
\label{section:RUC-regularity}
As discussed in Section \ref{section:Intro}, we need to regularize the variation of the local Lipschitz constant $\bL(\cdot)$ so as to balance the consensus error and descent progress. Denote $\mathcal{P}(\mathbb{R}^d)$ the family of all nonempty compact subsets 
of $\RR^d$, then the local smoothness constant of some objective function $\bL :\mathcal{P}(\RR^d)\to\RR_{++}$ is a positive-valued set function that maps any closed set $\cX\subseteq\RR^d$ to $\bL(\cX)$, (an upper bound of) the function's local smoothness constant on $\cX$. We introduce a relative uniform continuous (RUC) regularity condition for the local Lipschitz constant as a set function of the considered neighborhood. To start with, for any two compact convex sets $\cX,\cY\subseteq \RR^d$, we adopt the Hausdorff distance to measure their difference: 
$$d_{\rm H}(\cX,\cY):=\max\left\{\max_{x\in\cX}\min_{y\in\cY}\|x-y\|\,,\,\max_{y\in\cY}\min_{x\in\cX}\|x-y\|\right\},$$
see in \cite{birsan2006one}, which is different from the standard definition of the distance between two sets. We should note that the norm here is not necessarily the $\ell_2$-norm. When the variables are matrices, the norm $\|\cdot\|$ can be the spectral norm instead of the Frobenius norm. Based on this concept, we introduce the RUC regularity condition as follows.

\begin{definition}[RUC regularity]
\label{defn:RUC}
We say a positive-valued set function $F:\mathcal{P}(\mathbb{R}^d)\to\RR_{++}$ is relative uniform continuous, if $\forall\epsilon>0$, $\exists\delta >0$ s.t. the  relative function value difference satisfies  
$$\left|1-\max\left\{\frac{F(\cX)}{F(\cY)},\frac{F(\cY)}{F(\cX)}\right\}\right|\leq\epsilon
$$
for $\forall\cX,\cY\subseteq \mathcal{P}(\mathbb{R}^d)$ satisfying $d_{\rm H}(\cX,\cY)\leq \delta$.
\end{definition}
This condition can be viewed as a generalization of the classic uniform continuity concept to set functions and relative difference. Compared to the absolute difference, the use of relative difference allows a much more drastic change of function values. To illustrate this point, we can consider the counterpart of RUC for standard univariate function. Take the exponential function $f(x) = e^x$ for example. The standard uniform continuity property defined under the absolute difference requires:
$$\forall \epsilon>0, \exists \delta>0\quad \text{s.t.}\quad |f(x)-f(y)|\leq\epsilon\quad \text{for}\quad\forall |x-y|\leq \delta.$$
This property is clearly not satisfied by $e^x$ because $e^{x+\delta}-e^{x} = e^x(e^\delta-1)\to+\infty$ as $x\to+\infty$. 
That is, the uniform continuity under absolute function value difference is too restrictive and cannot describe the fast changing functions. However, if we switch to the relative difference for positive-valued functions, we have $|1-\max\{e^{x}/e^y,e^{y}/e^x\}| = e^{|x-y|}-1$, which implies that
$$\forall \epsilon>0, \exists \delta = \ln(1+\epsilon)>0 \quad\mbox{s.t.}\quad e^{|x-y|}-1<\epsilon\quad\mbox{for}\quad\forall |x-y|<\delta.$$
Hence $e^x$ is a univariate RUC function. By extending this observation to a general positive-valued set function 
$F:\mathcal{P}(\mathbb{R}^d)\to\mathbb{R}_{++}$, and by replacing the Euclidean 
distance between points with the Hausdorff distance between compact sets, we 
arrive at Definition~\ref{defn:RUC}.

Denote $\cF_{\text{ruc}}^d$ the class of all RUC functions over $\mathcal{P}(\RR^d)$, then this function class is closed under various common operations, which is summarized as Proposition \ref{proposition:closed}.

\begin{proposition}[Closedness of $\cF_{\text{ruc}}^d$]
\label{proposition:closed}
The RUC function class $\cF_{\emph{ruc}}^d$ is closed under affine composition, positive linear combination, maximum/minimum, and multiplication/division operations:\vspace{0.1cm} \\
\textbf{\emph{(1).}}
Let $F \in \cF_{\mathrm{ruc}}^d$ and let 
$\mathcal{A}:\mathbb{R}^{d'} \to \mathbb{R}^d$ be globally Lipschitz continuous. That is,   $\exists L>0$ s.t.
\[
    \|\mathcal{A}(x) - \mathcal{A}(y)\|
    \le L \|x-y\|,
    \qquad \forall x,y \in \mathbb{R}^{d'}.
\]
For $\cX \in \cP(\mathbb{R}^{d'})$ define 
$\mathcal{A}(\cX) := \{\mathcal{A}(x): x\in\cX\}$ and
$G(\cX) := F(\mathcal{A}(\cX))$.
Then $G \in \cF_{\mathrm{ruc}}^{d'}$.\\
\textbf{\emph{(2)}.}  If $F,G\in\cF_{\emph{ruc}}^{d}$, then $\alpha F+\beta G\in\cF_{\emph{ruc}}^{d}$ for any $\alpha,\beta>0$.\vspace{0.1cm}\\
\textbf{\emph{(3)}.} If $F,G\in\cF_{\emph{ruc}}^{d}$, then both $ \max\{F,G\}\in\cF_{\emph{ruc}}^{d}$ and $\min\{F,G\}\in\cF_{\emph{ruc}}^{d}$.\vspace{0.1cm}\\
\textbf{\emph{(4)}.} If  $F,G\in\cF_{\emph{ruc}}^{d}$, then both $F\cdot G\in\cF_{\emph{ruc}}^{d}$ and $F/G\in\cF_{\emph{ruc}}^{d}$.
\end{proposition}
\begin{proof}
First, for ease of notation, let us denote relative difference as 
$$d_{F}^{\rm rel}(\cX,\cY): = \left|1-\max\left\{\frac{F(\cX)}{F(\cY)},\frac{F(\cY)}{F(\cX)}\right\}\right|.$$
Then it is not hard to observe that $d_{F}^{\rm rel}(\cX,\cY)\leq \epsilon$ if and only if  $\max\left\{\frac{F(\cX)}{F(\cY)},\frac{F(\cY)}{F(\cX)}\right\}\leq 1+\epsilon$. First, let us  prove (1). By the $L$-Lipschitz continuity of $\mathcal{A}$, the induced set mapping is also
$L$-Lipschitz continuous under the Hausdorff distance:
\[
    d_H(\mathcal{A}(\cX),\mathcal{A}(\cY))
    \le L\cdot d_H(\cX,\cY),
    \qquad \forall \cX,\cY \in \mathcal{P}(\mathbb{R}^{d'}).
\]
Indeed, for any $x\in\cX$ we can find $y\in\cY$ with $\|x-y\|\le d_H(\cX,\cY)$,
and the Lipschitz property of $\mathcal{A}$ yields
$\operatorname{dist}(\mathcal{A}(x),\mathcal{A}(\cY))
\le L\cdot d_H(\cX,\cY)$, then taking maximum over all $x\in\mathcal{X}$ proves the first half of bound on Hausdorff distance between $\cA(\cX)$ and $\cA(\cY)$; the other half  follows by symmetry.
Next, fix an arbitrary $\epsilon>0$.
Because $F\in\cF_{\mathrm{ruc}}^d$, there exists $\delta>0$ such that
$d_H(\cX,\cY)\le\delta$ implies
$d_F^{\mathrm{rel}}(\cX,\cY)\le\epsilon$ for all
$\cX,\cY\in\mathcal{P}(\mathbb{R}^d)$.
Let $\delta' := \delta/L$.
Then for any $\cX,\cY\in\mathcal{P}(\mathbb{R}^{d'})$ with
$d_H(\cX,\cY)\le\delta'$, we have
\[
    d_H(\mathcal{A}(\cX),\mathcal{A}(\cY))
    \le L\, d_H(\cX,\cY)
    \le L\delta' = \delta,
\]
and hence, by definition that $G(\cX):=F(\mathcal{A}(\cX))$, we obtain
\[
    d_G^{\mathrm{rel}}(\cX,\cY)
    = d_F^{\mathrm{rel}}(\mathcal{A}(\cX),\mathcal{A}(\cY))
    \le \epsilon,
\]
which shows $G\in\cF_{\mathrm{ruc}}^{d'}$.  This completes the proof of (1).

Next we prove (2). Because $F,G\in\cF_{\rm ruc}^{d}$, for any $\epsilon>0$, there exist $\delta_F,\delta_G>0$ s.t. 
$d_{F}^{\rm rel}(\cX,\cY)\leq\epsilon$ for any $d_{\rm H}(\cX,\cY)\leq \delta_F$, and $d_{G}^{\rm rel}(\cX,\cY)\leq\epsilon$ for any $d_{\rm H}(\cX,\cY)\leq \delta_G$. Taking $\delta = \min\{\delta_F,\delta_G\}$, the above argument indicates that 
$F(\cX)\leq (1+\epsilon)F(\cY)$ and $G(\cX)\leq (1+\epsilon)G(\cY)$ simultaneously hold for any $d_{\rm H}(\cX,\cY)\leq \delta$, consequently
$$\frac{\alpha F(\cX)+\beta G(\cX)}{\alpha F(\cY)+\beta G(\cY)}\leq 1+\epsilon.$$
By the symmetry between $\cX$ and $\cY$,   repeating the above discussion with interchanged $\cX$ and $\cY$ gives 
$$\max\left\{\frac{\alpha F(\cX)+\beta G(\cX)}{\alpha F(\cY)+\beta G(\cY)},\frac{\alpha F(\cY)+\beta G(\cY)}{\alpha F(\cX)+\beta G(\cX)}\right\}\leq 1+\epsilon,$$
which implies an $\epsilon$-small relative difference: $d_{\alpha F+\beta G}^{\rm rel}(\cX,\cY)\leq\epsilon$. Hence we complete  the proof of (2) and conclude that $\alpha F+\beta G\in \cF_{\rm ruc}^{d}$.  

Next, we prove (3). For any $\epsilon>0$, let us inherit the choice of $\delta$ from the proof of (2). For any $d_{\rm H}(\cX,\cY)\leq\delta$, we assume without loss of generality that $F(\cX)\geq G(\cX)$. Then we have  
$$\frac{\max\{F(\cX),G(\cX)\}}{\max\{F(\cY),G(\cY)\}}=\frac{F(\cX)}{\max\{F(\cY),G(\cY)\}}\overset{(i)}{\leq}\frac{F(\cX)}{F(\cY)}\leq1+\epsilon,$$ where (i) is because $F(\cY)\leq\max\{F(\cX),G(\cX)\}$. Similar to the proof of (2), we can further use the symmetry between $\cX$ and $\cY$ to show $d_{\max\{F,G\}}^{\rm rel}(\cX,\cY)\leq\epsilon$ and $\max\{F,G\}\in\cF_{\rm ruc}^{d}$.

For the minimum function, we assume without loss of generality that $F(\cY)\geq G(\cY)$, then  
$$\frac{\min\{F(\cX),G(\cX)\}}{\min\{F(\cY),G(\cY)\}} = \frac{\min\{F(\cX),G(\cX)\}}{G(\cY)}\overset{(ii)}{\leq}\frac{G(\cX)}{G(\cY)}\leq 1+\epsilon,$$
where (ii) is due to $G(\cX)\geq\min\{F(\cX),G(\cX)\}$. Analogous to the maximum case, we can again use the symmetry argument to show $d_{\min\{F,G\}}^{\rm rel}(\cX,\cY)\leq\epsilon$ and hence $\min\{F,G\}\in\cF_{\rm ruc}^{d}$.

Finally, we prove (4). Given any fixed $\epsilon'>0$, let us inherit the choice of $\delta$ from the proof of (2) with  $\epsilon$ satisfying 
$(1+\epsilon)^2 = 1+\epsilon'$. Then direct computation already implies that $F\cdot G\in\cF_{\rm ruc}^{d}$, and we omit this simple and sheer computation for succinctness. The division is also simple. Note that the symmetric definition of relative difference indicates that $d_{\min\{G\}}^{\rm rel}(\cX,\cY) \equiv d_{\min\{1/G\}}^{\rm rel}(\cX,\cY)$ for any $\cX',\cY'\in\mathcal{K}(\RR^{d'})$. Hence $G\in\cF_{\rm ruc}^{d}$ indicates that $1/G\in\cF_{\rm ruc}^{d}$. Then $F/G\in\cF_{\rm ruc}^{d}$ will hold by viewing the division between $F$ and $G$ as the multiplication between $F$ and $1/G$.
\end{proof}

Based on the closedness property, if we can provide some basic RUC functions in $\cF_{\rm ruc}^{d}$, then one can construct new RUC-regular functions or verify RUC-regularity by affine composition, scaling, linear combination, or taking their piecewise variants by minimum or maximum operations. We notice that for many problems, the local smoothness constant $\bL(\cdot)$ grows in a  point-wise maximum form $\bL(\cX):=\max\{l(x): x\in\cX\}$ for some function $l(\cdot)$. In the next proposition, we provide a general sufficient condition for $l$ to guarantee the RUC-regularity of $\bL(\cdot)$.  
\begin{proposition} 
\label{proposition:growth}
Let $f:\mathbb{R}^d\to \mathbb{R}_{++}$ be a positive valued function. If
$\log(f)$ is uniformly continuous, then the set function $F:\mathcal{P}(\mathbb{R}^d)\to\mathbb{R}_{++}$ defined by $F(\mathcal{X}) := \max\{f(x):x\in\mathcal{X}\}$
is RUC-regular.
\end{proposition}

\begin{proof}

Fix $\epsilon>0$ and set $\eta=\log(1+\epsilon)$. By uniform continuity of
$\log(f)$, there exists $\delta>0$ such that $|\log f(x)-\log f(y)| \le \eta$ for all $\|x-y\|\le \delta$.  
Exponentiating this inequality yields 
\begin{equation}\label{eq:ratio-bound}
\frac{f(x)}{f(y)} \le e^\eta = 1+\epsilon
\quad\text{and}\quad
\frac{f(y)}{f(x)} \le 1+\epsilon, \qquad \forall \|x-y\|\le\delta.
\end{equation} 
Now take any $\mathcal{X},\mathcal{Y}\in\mathcal{P}(\mathbb{R}^d)$ such that 
$d_{\mathrm{H}}(\mathcal{X},\mathcal{Y})\le \delta$. By compactness of $\mathcal{X}$ and 
continuity of $f$, we can find $x^*\in\mathcal{X}$ s.t.
$F(\mathcal{X})=f(x^{*})$. By $d_{\mathrm{H}}(\mathcal{X},\mathcal{Y})\le \delta$, we can find some $y\in\mathcal{Y}$ s.t. $\|x^{*}-y\|\le\delta$. Applying
\eqref{eq:ratio-bound} gives
$\frac{F(\cX)}{F(\cY)} \leq \frac{f(x^*)}{f(y)} \leq 1+\epsilon$. Interchanging
$\mathcal{X}$ and $\mathcal{Y}$, we also obtain
$F(\mathcal{Y})/F(\mathcal{X})\le 1+\epsilon$. Therefore, $d^{\rm rel}_F(\cX,\cY)\leq \epsilon$, proving the RUC-regularity of $F$.

\end{proof}

Note that the (global) Lipschitz continuity or H\"older continuity implies the uniform continuity. Therefore,  proving these stronger continuity  conditions for $\log(f)$ could be a more quantitative way to establish the RUC-regularity of $F$. 

\begin{remark}
\label{remark:growth}
Let $f:\mathbb{R}^d\to\mathbb{R}_{++}$ be one of the following: \emph{(1)} $f(x) \equiv L$, \emph{(2)} $f(x) = \log(c+\|x\|)$, \emph{(3)} $f(x) = L + \|x\|^\nu$, \emph{(4)} $f(x) = e^{\|x\|/r}$, with $c>1$, $L>0$, $\nu>0$ and $r>0$. Then $\log(f)$ is Lipschitz continuous on $\mathbb{R}^d$, and the induced set function 
$F(\mathcal{X})=\max\{f(x):x\in\mathcal{X}\}$ is RUC-regular.
\end{remark}

\noindent For the succinctness of the presentation, we move the verification of this proposition and the remark to Appendix \ref{appdx:proposition:growth}. It can be observed that the proof of the proposition only involves triangle inequality of the norm $\|\cdot\|$ and the basic calculus, and is hence not limited to the $\ell_2$-norm. For example, for matrix functions, it can also be the spectral norm. Nevertheless, one can always use the Frobenius norm for simplicity. 

\begin{example}
\label{example:matrix-decomp}
Let $A\in\RR^{m\times n}$ be the input data matrix and let  $X\in\RR^{m\times k}$,  $Y\in\RR^{n\times k}$ be two matrix decision variables.
For the matrix decomposition problem with objective function $f(X,Y) = \|A-XY'\|_F^2$, within any $\cZ\in\mathcal{P}(\RR^{m\times k}\times\RR^{n\times k})$ , its local Lipschitz constant will be 
$$\bL(\cZ):=\max_{[X;Y]\in\cZ}\,\,\|A\|+ 3\cdot\max\left\{\|X\|^2,\|Y\|^2\right\}.$$
Let $F(\cZ) =\max_{Z\in\cZ} \frac{1}{3}\|A\| + \|Z\|^2$, and denote $\cA_X, \cA_Y$ the linear operators that map $Z$ to its $X$ and $Y$ components, respectively. Then we know 
$$\bL(\cZ)=3\cdot\max\left\{F\circ\mathcal{A}_X(\cZ),F\circ\mathcal{A}_Y(\cZ)\right\}.$$
Propositions \ref{proposition:closed} and \ref{proposition:growth} indicates the RUC-regularity of this function. 
\end{example}
\noindent The proof of Example \ref{example:matrix-decomp} only consists of basic calculus and is omitted for simplicity. 
\begin{example}
\label{example:gen-smooth}
Suppose $f:\RR^d\mapsto\RR$ is an $\alpha$-generalized smooth function such that $\|\nabla^2 f(x)\| \leq L_0 + L_1\|\nabla f(x)\|^\alpha$ for some $L_0,L_1>0$ and $\alpha\in(0,1)$. Given any reference point $\bar{x}$, then the Hessian matrix satisfies
$$\|\nabla^2f(x)\|\leq A_0(\bar{x})+A_1(\bar{x})\|x-\bar{x}\|^{\frac{\alpha}{1-\alpha}},$$
where the constants $A_0$ and $A_1$ are defined by 
$$\begin{cases}
A_0(\bar{x}) := L_0 
+  2^{\tfrac{\alpha}{1-\alpha}}L_1\|\nabla f(\bar{x})\| 
+  2^{\tfrac{\alpha^2}{1-\alpha}} L_0 L_1,\\
A_1(\bar{x}):= 2^{\tfrac{\alpha^2}{1-\alpha}} L_0 L_1
+ 2^{\tfrac{\alpha^2}{1-\alpha}}L_1\left((1-\alpha)\,L_1\right)^{\tfrac{\alpha}{1-\alpha}}.
\end{cases}$$
Therefore, the local smoothness constant of $f$ can be defined as 
$$\bL(\cX) := A_0(\bar{x}_i)+A_1(\bar{x}_i)\max_{x\in\cX}\|x-\bar{x}_i\|^{\frac{\alpha}{1-\alpha}}.$$
By Proposition \ref{proposition:growth}, we know $\bL$ is a RUC-regular function.
\end{example} 
\noindent The proof of this example is placed in Appendix \ref{appdx:example:gen-smooth}.

\section{Clipped gradient tracking with staggered variance reduction}
\label{section:CGTVR} 

In this section, we present the clipped gradient tracking method with staggered variance reduction, as well as its simplified synchronous version. We call them \texttt{CGTVR-stag} and \texttt{CGTVR-sync} to distinguish the two variants. Before presenting the algorithms, we would like to make the following assumptions about the objective function of the distributed optimization problem \eqref{prob:main}, based on our previous discussion of RUC local smoothness constants. 

\begin{assumption}
\label{assumption:RMSS} 
For each node $i\in[m]$ and any $\cX\in\mathcal{P}(\RR^d)$, we assume the finite-sum objective function $f_i$ to be $\bL_i(\cX)$-mean-squared-smooth on $\cX$, that is
\begin{equation} 
     \frac{1}{n_i}\sum_{j=1}^{n_i} \|\nabla \fij(x)-\nabla \fij(x') \|^2  \leq \bL_i^2(\cX)\|x-x'\|^2, \quad \forall x, x' \in \cX.\nonumber
\end{equation}  
As local smoothness constants, we further assume that each $\bL_i$ is monotone and additive in the sense that $\bL_i(\cX)\leq\bL_i(\cY)$ for any $\cX\subseteq\cY$, and $\max\{\bL_i(\cX),\bL_i(\cY)\} = \bL_i(\cX\cup\cY)$. Without loss of generality, we
assume that $\bL_i(\cX)\geq1$，for any agent $i\in[m]$ and any set $\cX$. Otherwise, we can make the following modification $\bL_i(\cX) \leftarrow \max\{\bL_i(\cX),1\}$.
\end{assumption}
\noindent As the local smoothness constants of some locally smooth function $f_i$, the monotone and additive properties assumed above are very natural, and is satisfied by all the examples in this paper. 
Throughout the paper, we default that the function $\bL_i(\cdot)$ is known to all other nodes over the network.  

\begin{assumption}
    \label{assumption:RUC}
    For each node $i\in[m]$, we assume $\bL_i\in\cF_{\rm ruc}^d$. 
\end{assumption}
\noindent Denote the maximum local smoothness constant among all agents as
\begin{equation}
    \label{eqn:max-L}
    \bL := \max\left\{\bL_1,\cdots,\bL_m\right\},
\end{equation}
then Proposition \ref{proposition:closed} immediately implies $\bL\in\cF_{\rm ruc}^d$ it is a straightforward exercise that $\bL$ also inherit the monotone and additive properties from each $\bL_i$. And it is straightforward that the $f = \frac{1}{m}\sum_{i=1}^mf_i$ is $\bL(\cX)$-smooth over the convex set $\cX$. 
\begin{assumption}[Finite infimum]
\label{ass:existence-minimizer}
The objective function  of problem~\eqref{prob:main} is bounded from below. That is, we assume $f^* = \inf\{f(x):x\in\mathbb{R}^d\}>-\infty.$
\end{assumption}

Finally, for the communication network of the distributed optimization problem, with node set $\mathcal{V}=[m]$ and undirected edge set $\mathcal{E}$, we assume the associated mixing matrix $W$ to satisfy the following assumption.  
 
\begin{assumption}
\label{assumption:mix-spectral}
Based on the network structure, the mixing matrix $W$ satisfies $W_{ij}=0$ for all $(i,j)\notin\mathcal{E}$, and $W_{ij}>0$ otherwise. In addition, we require $W\mathbf{1}=W^\top\mathbf{1}=\mathbf{1}$,  and $\underline{\lambda}_{\max }(\mathbf{W})=\eta\in(0,1)$ where $\underline{\lambda}_{\max}(\cdot)$ denotes the second largest eigenvalue. 
\end{assumption} 
\noindent As a direct result of Assumption \ref{assumption:mix-spectral}, it holds that
\begin{equation}
    \label{eqn:mix-spectral}
    \Big\|\mathbf{W}-\frac{1}{m} \mathbf{1}\mathbf{1}^\top\Big\| = \eta < 1.
\end{equation}
Setting $\epsilon = \min\left\{1,\frac{1-\eta}{2\eta}\right\}$ and $\gamma = 1+\epsilon = \min\left\{2,\frac{1+\eta}{2\eta}\right\}$, then Assumption \ref{assumption:RUC} implies that 
\begin{equation}
    \label{eqn:ruc-r0}
    \exists r_0>0\qquad\mbox{s.t.}\qquad \gamma^{-1}\leq \frac{\bL_i(\cX)}{\bL_i(\cY)} \leq \gamma\qquad\mbox{for}\qquad\forall i\in[m],\,\,\,\, \forall d_{\rm H}(\cX,\cY)\leq r_0,
\end{equation} 
which, by Proposition \ref{proposition:closed}, also holds for the maximum local smoothness constant $\bL(\cdot)$ with the same set of parameters $r_0$ and $\gamma$. 

Given the constants $\eta,r_0$, and let $\bar{n}=\frac{1}{m}\sum_{i=1}^mn_i$ be the average data size among the agents. Then it is straightforward to design the clipped variant of the SARAH/SPIDER based variance reduced GT algorithms \cite{sun2020improving,xin2020near} that can achieve the optimal $O(\bar{n}^{1/2})$ dependence under classic Lipschitz smoothness setting. Let us present this vanilla extension as \texttt{CGTVR-sync} in Algorithm \ref{algorithm:TVRG} as a warm up of the more efficient but also more complicated \texttt{CGTVR-stag} method.  

\begin{algorithm2e}[H]
\caption{Clipped GT with Synchronous Variance Reduction (\texttt{CGTVR-sync})} 
\label{algorithm:TVRG}
\textbf{Input:} Initialize $x_{1}^{0} = \cdots = x_{m}^{0} = x^0$, clipping threshold $\delta$, epoch length $\tau$,  and constants $\theta,r_0$. \\
Each agent $i\in[m]$, default  $y_{i}^{-1} = g_{i}^{-1} = 0$\\
\For{$t=0,1,\cdots,T-1$, \emph{each agent} $i$}{
    \uIf{$t\!\!\mod\tau == 0$}{
     Compute gradient at check point $\git = \nabla f_i(\xit)$. \\
     Compute local constant $L_{i}^{t}:= \bL\left(\cX_{i}^{t}\right)$ for the ball area $\cX_{i}^{t}:=B(\xit,r_0)$.\\
     All agents broad-cast local constant and compute $\beta_t = \theta L_t$ with $L_t:=\max_j L_{j}^{t}$. \\ 
    }
    \Else{Sample a batch $\cB_i^t$ and compute:
    $\git = \gitm + \frac{1}{|\cB_i^t|}\sum_{j\in\cB_i^t}\!\left(\nabla\fij(\xit)-\nabla\fij(\xitm)\right).$\\
    Inherit the previous parameter $\beta_t = \beta_{t-1}$.  
    }
    Update gradient tracker: $\,\,\yit = \sum_{j}W_{ij}y_j^{t-1} + \git - \gitm.$\\
    Update decision variable: $x_i^{t+1} = \sum_{j}W_{ij}x_j^t - \ctdb\big(y_i^t\big)$.  
    }
\end{algorithm2e}
\vspace{0.1cm}

In Algorithm \ref{algorithm:TVRG}, \texttt{CGTVR-sync} initializes the $(s+1)$-th epoch, by determining the active region $\cX_{i}^{s\tau}$ of the next $\tau$ iterations for each agent $i$, where the region $\cX_{i}^{s\tau}$ is an $r_0$-radius ball centered at the initial point $x_{i}^{s\tau}$ of the current epoch. Then the agents can estimate their local smoothness constant $L_i^{t}=\bL(\cX_{i}^{s\tau})$ for $s\tau\leq t\leq (s+1)\tau-1$ in this region and then communicate to get maximum local smoothness constant $L_{t} = \max_i L_i^{t}$ to construct the stepsize $1/\beta_t$, which will remain unchanged within the current epoch. Then we update the gradient estimator $\git$ by the standard SARAH/SPIDER scheme and update the gradient tracker $y_i^t$ by the standard GT scheme. To ensure the next $\tau$ local iterates of agent $i$ to stay in $\cX_{i}^{s\tau}$, we introduce a norm-based clipping operator for any update direction $v\in\RR^d$ and any stepsize parameter $1/\beta>0$: 
\begin{align}
\label{defn:clip-operator}
\mathcal{T}_\beta^\delta(v):= \begin{cases}
    v/\beta,& \mbox{if }\|v\|\leq \beta\delta,\\
    {\delta\cdot v}/{\|v\|}, &\mbox{otherwise}.
\end{cases}
\end{align}
Then we adopt the clipped gradient update in Line 12:
\begin{equation}
    \label{eqn:clipped-update}
    x_i^{t+1} = \sum_{j}W_{ij}x_i^t - \ctdb \big(y_i^t\big),
\end{equation}
where the gradient update is guaranteed to satisfy $\|\ctdb \big(y_i^t\big)\|\leq \delta$ due to the norm-based clipping scheme. By the mixing property of Assumption \ref{assumption:mix-spectral}, it can be shown that 
\begin{equation}
    \label{eqn:sync-epoch-radii}
    \big\|x_{i}^{s\tau}-x_i^{s\tau+k}\big\|\leq \left(\frac{2\sqrt{2m}}{1-\eta}+k\right)\cdot \delta \qquad\mbox{for}\qquad\forall 1\leq k\leq \tau,
\end{equation}
regardless of the randomness in gradient estimators. Let us denote the constant $c = \frac{2\sqrt{2m}}{1-\eta}$ for simplicity because it frequently appears in latter analysis. Then the above bound immediately indicates that $\{x_{i}^{s\tau+k}\}_{k=0}^{\tau} \subseteq \cX_{i}^{s\tau}$ as long as we select a small enough clipping threshold $ \delta \leq \frac{r_0}{c+\tau}$. Therefore, the estimated local smoothness constant $L_{t}$, with $t/\tau\in[s,s+1)$, is actually evaluating
$$L_t = \max_{1\leq i\leq m}\,\,\bL(\cX_{i}^{s\tau}) = \bL(\cX_{s\tau})\qquad\mbox{with}\qquad \cX_{s\tau} := \bigcup^m_{i=1}\cX_{i}^{s\tau}.$$
Clearly, the Hausdorff distance between $\cX_{s\tau}$ and $\cX_{(s+1)\tau}$ satisfies
$$d_{\rm H}\big(\cX_{s\tau},\cX_{(s+1)\tau}\big)\leq\max_{1\leq i\leq m}d_{\rm H}\big(\cX_{i}^{s\tau},\cX_{i}^{(s+1)\tau}\big) = \max_{1\leq i\leq m}\big\|x_{i}^{s\tau}-x_i^{(s+1)\tau}\big\| \leq r_0.$$
Hence \eqref{eqn:ruc-r0} is activated and the sequence of estimated local smoothness constants $\{L_t\}_{t\geq0}$ will only change slowly in the relative sense, hence facilitating the analysis of gradient tracking methods. 

Compared to the existing SARAH/SPIDER based variance reduced GT methods \cite{sun2020improving,xin2020near}, \texttt{CGTVR-sync} adopts a straightforward modification by replacing a global smoothness constant $L$ with a sequence of estimated local smoothness constants $L_t$. To ensure the effectiveness of the local estimates $L_t$, we propose to additionally incorporate a clipped gradient update \eqref{eqn:clipped-update} with small enough $\delta$. 

However, let us simplify the discussion by adopting the assumption that $n=n_1=\cdots=n_m$ from \cite{sun2020improving,xin2020near}. Then the typical epoch length $\tau$ for SARAH type variance reduction scheme is $\tau\in[\sqrt{n},n]$, see \cite{pham2020proxsarah}. When $n$ is excessively large, the requirement that $\delta\leq\frac{r_0}{c+\tau}$ will result in an extremely small clipping threshold, leading to overly conservative steps and slow practical performance. In order to resolve this issue, we propose a more flexible staggered variant of \texttt{CGTVR-sync}, where each agent no longer needs to broad-cast and communicate their locally estimated smoothness constant, and no need to take the $\cO(r_0/\tau)$ conservative steps. We present the improved variant as Algorithm \ref{algorithm:ATVRGT}, and we call it the clipped GT with staggered variance reduction (\texttt{CGTVR-stag}). 

\begin{algorithm2e}[H]
\caption{Clipped GT with Staggered Variance Reduction (\texttt{CGTVR-stag})} 
\label{algorithm:ATVRGT}
\textbf{Input:} Initialize $x_{1}^0 = x_2^0 = \cdots = x_m^0$, clipping threshold $\delta$, epoch radius $d$,  constants $\theta$ and ${r_0}$, and the maximum epoch length $\tau_i$ for each agent $i\in[m]$. \\
Each agent $i\in[m]$, default  $y_i^{-1} = g_i^{-1} = 0$, ${\rm Flag}(i)={\bf true}$ and ${\rm Count}(i) =0$.\\
\For{$t=0,1,2,\cdots,T-1$, }{
\For{\emph{each agent} $i$}{\uIf{${\rm Flag}(i)$}{
    Compute gradient at check point $\git = \nabla f_i(x_i^t)$ and record the check point ${x}_i^{\rm c}\leftarrow x_i^t$. \\
    Compute local constant $L_{i}^t:= \bL\left(\cX_{i}^{t}\right)$ in the ball $\cX_{i}^{t}:=B(x_{i}^t,{r_0})$  and set $\beta_i^t = \theta L_i^t$.\\

    Update epoch restart signal ${\rm Flag}(i)\leftarrow {\bf false}$ and iteration counter ${\rm Count}(i)\leftarrow 0$.
 }
    \Else{Sample a batch $\cB_i^t$ and compute:
    $\git = \gitm + \frac{1}{|\cB_i^t|}\sum_{j\in\cB_i^t}\!\left(\nabla\fij(x_i^t)-\nabla\fij(x_i^{t-1})\right).$\\
    Inherit the previous inverse step size $\beta_i^t = \beta_i^{t-1}$.  
    }
    Update gradient tracker: $\,\,y_i^{t} = \sum_{j}W_{ij}y_j^{t-1} + \git - \gitm.$\\
    Update decision variable: $x_i^{t+1} = \sum_{j}W_{ij}x_j^t - \ctdbi\big(y_i^t\big)$. \\
    \uIf{$\|x_i^{t+1}-{x}_i^{\rm c}\| \geq d $ or ${\rm Count}(i)  == \tau_i-1$}{
    Renew epoch restart signal: ${\rm Flag}(i)\leftarrow {\bf true}$.} 
    \Else{Increase local iteration counter: ${\rm Count}(i)\leftarrow {\rm Count}(i) + 1$.}}    
    }  
\end{algorithm2e}
\vspace{0.1cm}

Compared to \texttt{CGTVR-sync}, a naive extension of \cite{sun2020improving,xin2020near}, a major difference in \texttt{CGTVR-stag} lies in the early stop strategy in Line 14 of Algorithm \ref{algorithm:ATVRGT} that allows the local agent $i$ to terminate the epoch before running the full $\tau_i$ steps if the iterations have already made enough progress, which is represented by a large enough update ($\|x_i^{t+1}-{x}_i^{\rm c}\| \geq d$) compare to the initial solution (check point ${x}_i^{\rm c}$) of the current epoch. Because the agent need not run the full $\tau_i$ steps of the epoch, we no longer need to use the pessimistic estimate of the iterates' active region \eqref{eqn:sync-epoch-radii}, and hence it will be sufficient to choose $\delta  = \cO(r_0/c) = \cO(1)$ for \texttt{CGTVR-stag} rather than the conservative $1/n\sim1/\sqrt{n}$ clipping threshold in \texttt{CGTVR-sync}. Therefore, this strategy allows a much more aggressive update and can improve practical performance for many problem instances.  As a consequence of early stopping, each agent will naturally experience staggered, or asynchronous, initialization of the epochs, which makes the synchronous broad-cast and take max step (Algorithm \ref{algorithm:TVRG}, Line 6) unnecessary. In \texttt{CGTVR-stag}, each agent simply utilizes the locally estimated constant $L_i^t$ and the clipped update as well as the early stopping criteria will guarantee their relative difference $|1-\max\{ {L_i^t}/{L_j^t}, {L_j^t}/{L_i^t}\}|$ to be small. To showcase this difference, we provide the Lemma \ref{lemma:clip-distance2mean} and Proposition \ref{proposition:stager-dist-bound} about the clipped updates. To simplify the presentation of these results and latter discussion, we adopt the following block notations: 
\begin{equation}
    \label{defn:block-notation}
     \bxt = \begin{bmatrix}
    (x_1^t)^\top\\
    \vdots\\
    (x_m^t)^\top
\end{bmatrix}, \qquad \byt = \begin{bmatrix}
    (y_1^t)^\top\\
    \vdots\\
    (y_m^t)^\top
\end{bmatrix}, \qquad \bGt = \begin{bmatrix}
    ({g}_1^t)^\top\\
    \vdots\\
    ({g}_m^t)^\top
\end{bmatrix}, \qquad\mbox{and}\qquad\ctdb(\byt)=\begin{bmatrix}
    \mathcal{T}^\delta_{\beta_1^t}(y_1^t)^\top\\
    \vdots\\
    \mathcal{T}^\delta_{\beta_m^t}(y_m^t)^\top
\end{bmatrix}, 
\end{equation}
where $\bxt,\byt\in\RR^{m\times d}$ are constructed by stacking the local variable $\xit$ and the local gradient tracker $\yit$ row by row, respectively.   The construction of $\bGt$ and $\ctdb(\byt)$ are similar. In particular, for the notation $\ctdb(\byt)$, we use $\beta_t:=\{\beta_1^t,\beta_2^t,\cdots,\beta_m^t\}$ to denote the collection of all local $\beta_i^t$ for simplicity. Under this notation, the clipped update in \texttt{CGTVR-sync} and \texttt{CGTVR-stag} can be compactly written as 
\begin{eqnarray}
\label{eqn:CGT}
    \begin{cases}
        \bxte = \bW\bxt - \ctdb\left(\byt\right),\\
        \byte = \bW\byt + \bGte- \bGt.
    \end{cases}
\end{eqnarray}
In addition, define the (transposed) average variable and gradient tracker as the following row vectors:
$$\bar{\mathbf{x}}^t:=\frac{1}{m}\mathbf{1}^\top\bxt=\frac{1}{m}\sum_i (\xit)^\top, \quad \bar{\mathbf{y}}^t:=\frac{1}{m}\mathbf{1}^\top\byt=\frac{1}{m}\sum_i (\yit)^\top, \quad \bar{\mathbf{G}}^t = \frac{1}{m}\mathbf{1}^\top \mathbf{G}^t = \frac{1}{m}\sum_i ({g}_i^t)^\top, $$
where ${\bf 1}$ denotes the all-one column vector. Now we state the following results for the clipped update.

\begin{lemma}
    \label{lemma:clip-distance2mean}
    Suppose the mixing matrix $\bW\in\RR^{m\times m}$ satisfies Assumption \ref{assumption:mix-spectral}, and the sequence $\{\bxt\}_{t\geq0}\subseteq\RR^{m\times d}$ is generated by $\bxte = \bW\bxt + \bvt$ from a consensus initialization $x_1^0 =x_2^0= \cdots = x_m^0$ and an arbitrary sequence $\{\bvt\}_{t\geq0}\subseteq\RR^{m\times d}$ such that $\|\bvt\|_{2,\infty}\leq \delta$ for some fixed $\delta>0$ and $\forall t\geq0$. Then for the sequence $\{\bxt\}_{t\geq0}$, the following relationships hold for all $t\geq0$ that
    $$\lnorm\bbxte-\bbxt\rnorm \leq \delta\qquad\mbox{and}\qquad\lnorm\bxt-{\bf 1}\bbxt\rnorm_{F}\leq {c\delta}/{2},$$
    where $c = \frac{2\sqrt{2m}}{1-\eta}$ is a constant.  
\end{lemma}
\begin{proof}
    First, due to the fact that $\bW^\top{\bf 1}={\bf 1}$, it always hold that $\bbxte = \bbxt + \bbvt.$ Hence,  we have
    $$\left\|\bbxte-\bbxt\right\| = \left\|({\bf 1}/m)^\top \bvt\right\| \leq \left\|\bvt\right\|_{2,\infty}\cdot\|{\bf 1}/m\|_1 \leq \delta.$$
    Second, with the constant $\zeta = 1-\eta\in(0,1)$, we have 
    \begin{eqnarray*}
        \lnorm\bxte-\bone\bbxte\rnorm^2_F &=& \lnorm\left(\bW-\bone\bone^\top/m\right)\left(\bxt-\bone\bbxt\right) + \left(I-\bone\bone^\top/m\right)\bvt\rnorm_F^2 \\
        & \leq & (1+\zeta)\lnorm \bW-\bone\bone^\top/m\rnorm^2\lnorm\bxt-\bone\bbxt\rnorm_F^2 + (1+1/\zeta)\lnorm I-\bone\bone^\top/m\rnorm^2\lnorm\bvt\rnorm_F^2.
    \end{eqnarray*}   
    By \eqref{eqn:mix-spectral} and using our special choice of $\zeta = 1-\eta$ as well as the condition that $\|\bvt\|^2_F\leq m\delta^2$, we obtain
    $$\lnorm\bxte-\bone\bbxte\rnorm^2_F\leq (2-\eta)\eta^2\lnorm\bxt-\bone\bbxt\rnorm_F^2 + \frac{2-\eta}{1-\eta}\lnorm\bvt\rnorm_F^2 \leq  \eta\lnorm\bxt-\bone\bbxt\rnorm_F^2 + \frac{2m\delta^2}{1-\eta},$$
    which further indicates that 
    $\|\bxt-\bone\bbxt\|^2_F \leq \frac{2m\delta^2}{(1-\eta)^2}= c^2\delta^2/4$ for all $t\geq0$ given the consensus initialization that $x_1^0=x_2^0=\cdots=x_m^0$, which completes the proof of the lemma. 
\end{proof}
Note that in the clipped gradient tracking update \eqref{eqn:CGT}, the adoption of the clipping operator \eqref{defn:clip-operator} guarantees the update to satisfy $\|-\ctdb(\byt)\|_{2,\infty}\leq \delta$  regardless of $\beta_t$ and $\byt$. Therefore, Lemma \ref{lemma:clip-distance2mean} directly applies to the iteration sequence $\{\bxt\}$ generated by \texttt{CGTVR-sync} and \texttt{CGTVR-stag}. Because the Frobenius norm always upper bounds the $\ell_2$-$\ell_\infty$ norm, Lemma \ref{lemma:clip-distance2mean}  indicates that all the local iterates remains $(c\delta/2)$-close to the average iterate:
$\|x_i^t - (\bbxt)^\top\| \leq c\delta/2$ for all $i\in[m]$ and $t\geq0$. As a direct consequence, we can upper bound the distance of any two local iterates by 
\begin{equation}
    \label{eqn:local-iter-dist}
    \|x_i^t-x_j^{t'}\| \leq (c+|t'-t|)\delta
\end{equation}
for any $i,j\in[m]$ and $t,t'\geq0$, please refer to Figure \ref{fig:CGT-dist}-(a). This complement the earlier discussion of \texttt{CGTVR-sync} in this section. 

For \texttt{CGTVR-stag} where all the agents asynchronously initialize or early stop their local epochs, at any time point $t\geq0$, let us define the mapping $u$ that returns the last time step right before terminating the current epoch at agent $i$ as $u(i,t)$. Similarly, we define the mapping $l$ to return the first time step of the current epoch at agent $i$ as $l(i,t)$. Based on this definition, if agent $i$ terminates the current epoch at time $t$, then we should write $t = l(i,t)$ as it is the first time step of the new epoch. Therefore, we should notice that the adoption of an early stop criterion may cause the epoch to be shorter than $\tau_i$ steps, that is
$$u(i,t)-l(i,t)\leq \tau_i, \qquad \forall i\in[m],\,\,\forall t\geq0$$
while the equality may not necessarily be attained. In this case, bounding the relative similarity between the local smoothness constants $L_i^t,L_j^t$ at two arbitrary agents $i,j$ essentially reduces to bounding the distance between $x_i^{l(i,t)}$ and $x_j^{l(j,t)}$, which is described by the following proposition.  
\begin{proposition}
    \label{proposition:stager-dist-bound}
    Let the sequence $\{\bxt\}_{t\geq0}$ be generated by \texttt{CGTVR-stag} (Algorithm \ref{algorithm:ATVRGT}), then 
    $$d_{\rm H}\left(\cX_i^{l(i,t)},\cX_j^{l(j,t)}\right) \leq c\delta+d  \quad\mbox{and}\quad d_{\rm H}\left(\cX_i^{l(i,t)},\cX_j^{l(j,t+1)}\right)\leq (c+1)\delta+d$$
    For all $i,j\in[m]$ and any $ t\geq0.$
    Suppose Assumptions \ref{assumption:RMSS} and \ref{assumption:RUC} hold. Then if we select $\delta$ and $d$ such that $(c+1)\delta+d\leq r_0$, we can guarantee  
    $$\frac{1}{\gamma}\leq  \max\left\{\frac{L_j^t}{L_i^t}\,\,,\, \frac{L_j^{t+1}}{L_i^t}\right\}\leq \gamma$$
    for any $i,j\in[m]$ and $t\geq0$, where $\gamma = \min\left\{2,\frac{1+\eta}{2\eta}\right\}$ is defined in \eqref{eqn:ruc-r0}.
\end{proposition}
\begin{proof}
    Lemma \ref{lemma:clip-distance2mean} shows that the average iterates $\bxt$ provides a bridge to bound the distance between different local iterates. One only needs to use such a bridge in triangle inequality to bound the distance between the check points $x_i^{l(i,t)}$ and $x_j^{l(j,t)}$.  By the early stop criterion (Line 14, Algorithm \ref{algorithm:ATVRGT}), we know that all the local iterates for agent $i$ and agent $j$ satisfy the following relationship:
    $$\left\{x_i^k: l(i,t)\leq k\leq t\right\}  \subseteq B\left(x_i^{l(i,t)},d\right)\qquad\mbox{and}\qquad \left\{x_j^k: l(j,t)\leq k \leq t\right\} \subseteq B\left(x_j^{l(j,t)},d\right).$$
    See illustration in Figure \ref{fig:CGT-dist}-(b). Now, let us assume w.l.o.g. that $l(j,t)\leq l(i,t)$. Then we can use $\bar{\bf x}^{l(i,t)}$ and $x_j^{l(i,t)}$ as bridges and use Lemma \ref{lemma:clip-distance2mean} to yield
    $$\lnorm x_i^{l(i,t)}-x_j^{l(j,t)}\rnorm \leq \lnorm x_i^{l(i,t)}-\bar{\bf x}^{l(i,t)}\rnorm + \lnorm \bar{\bf x}^{l(i,t)}-x_j^{l(i,t)}\rnorm + \lnorm x_j^{l(i,t)}-x_j^{l(j,t)}\rnorm \leq c\delta+d.$$
    Then it remains to use the fact that 
    $d_{\rm H}\big(\cX_i^{l(i,t)},\cX_j^{l(j,t)}\big) = \big\|x_i^{l(i,t)}-x_j^{l(j,t)}\big\|$ to prove the first bound. 
    
    For the second bound, we consider two cases. In the first case, the agent $j$ does not terminate the current epoch, then $x_j^{l(j,t+1)} = x_j^{l(j,t)}$ does not change and we still have $d_{\rm H}\big(\cX_i^{l(i,t)},\cX_j^{l(j,t)}\big) \leq c\delta+d$. In the second case, the agent $j$ does terminate the current epoch and start a new one at time $t+1$. Then in this case, $x_j^{l(j,t+1)} = x_j^{t+1}$ changes to a new point. This case is illustrated in Figure \ref{fig:CGT-dist}-(c), where we use the chain $x_i^t$, $\bbxt$, $\bbxte$ as the bridge to obtain
    $$\lnorm x_i^{l(i,t)}-x_j^{l(j,t+1)}\rnorm \leq \lnorm x_i^{l(i,t)}-x_i^t\rnorm + \lnorm x_i^t-\bbxt\rnorm + \lnorm \bbxt-\bbxte\rnorm + \lnorm\bbxte-x_j^{t+1}\rnorm\leq (c+1)\delta+d.$$
    Combining the two cases proves the second Hausdorff distance bound.  
    Then the remaining argument of the proposition directly follows Assumptions \ref{assumption:RMSS}, \ref{assumption:RUC}, and \eqref{eqn:ruc-r0}. 
\end{proof} 

\begin{figure}[h!]
    \centering 
    \begin{subfigure}{0.27\textwidth}
        \centering
        \includegraphics[width=\textwidth]{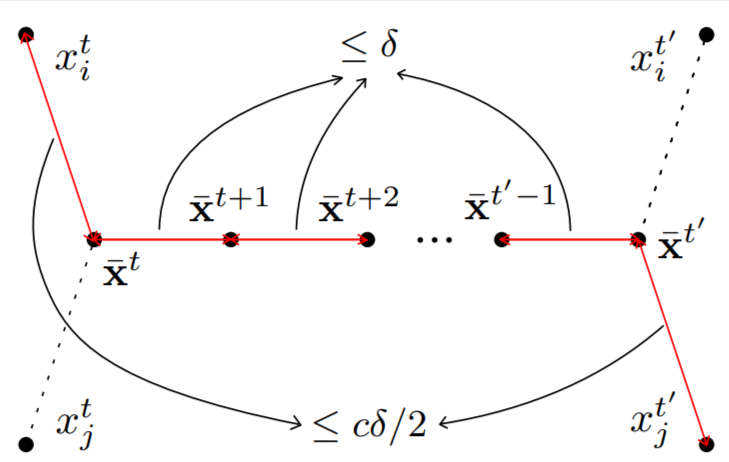}   
        \caption{ }
    \end{subfigure} 
    \begin{subfigure}{0.23\textwidth}
        \centering
        \includegraphics[width=\textwidth]{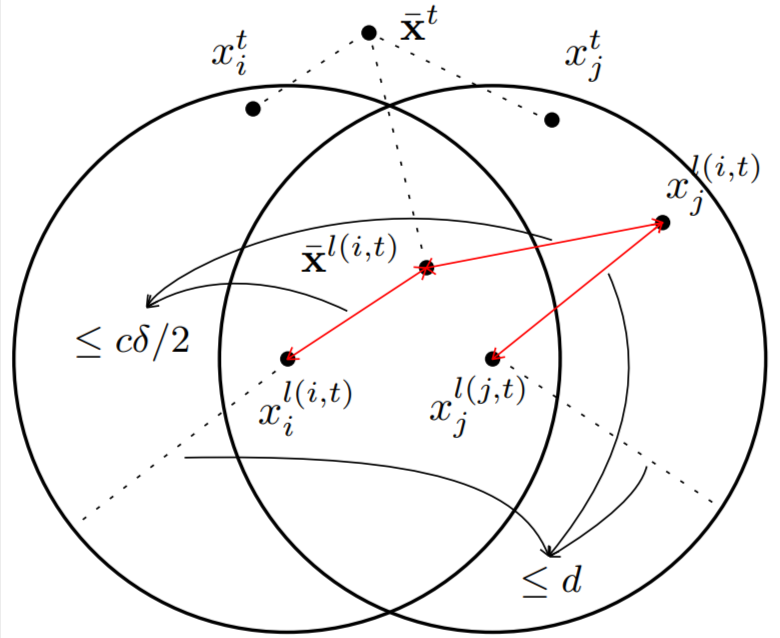}  
        \caption{ }
    \end{subfigure} 
    \begin{subfigure}{0.25\textwidth}
        \centering
        \includegraphics[width=\textwidth]{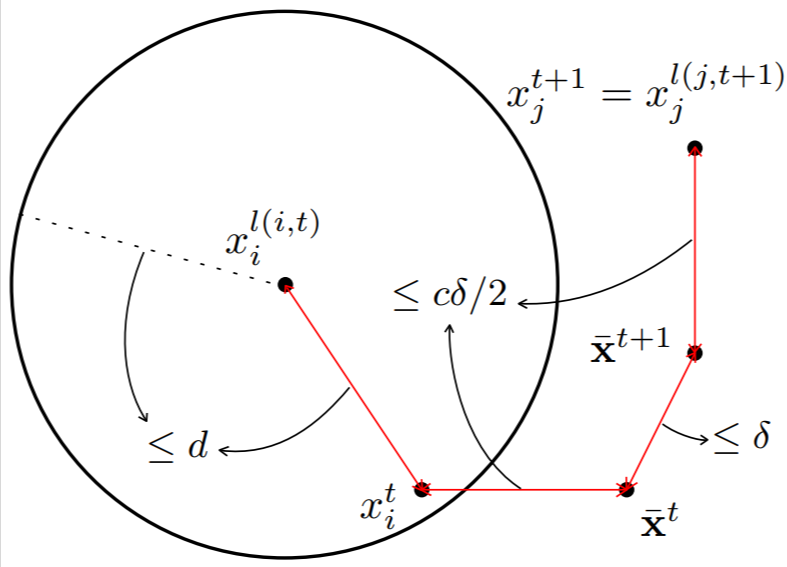}  
        \caption{ }
    \end{subfigure}
    \begin{subfigure}{0.22\textwidth}
        \centering
        \includegraphics[width=\textwidth]{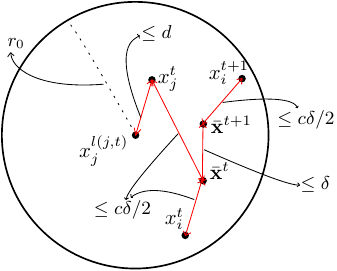}  
        \caption{ }
    \end{subfigure}
    \caption{(a) illustrates the distance upper bound for two arbitrary local iterates at time $t$ and $t'$. (b) illustrates the difference between the local check points of arbitrary two agents at any time point $t$. (c) illustrates the difference between the local check points of arbitrary two agents at any time point $t$ and $t+1$. (d) illustrates the membership relationship between several important points and the set $\cX_j^{l(j,t)}:=B(x_j^{l(j,t)},r_0)$ for an arbitrary agent $j\in[m]$.}
    \label{fig:CGT-dist}
\end{figure}

\begin{remark} \label{remark:Stag-point-membership}
Based on the analysis of Lemma \ref{lemma:clip-distance2mean} and Proposition \ref{proposition:stager-dist-bound}, we know that for any $t\geq 0$ and any $i,j\in[m]$, it holds that 
$\big\{\bxt, \bxte, \xit, \xite\big\} \subseteq \mathcal{X}_j^{l(j,t)}$.
\end{remark}
The proof of this remark is illustrated in Figure \ref{fig:CGT-dist}-(d), and is hence omitted for succinctness, and the requirement that on $\delta,d$ and $r_0$ is exactly the same as that of Proposition \ref{proposition:stager-dist-bound}. With this observation in mind, we know that although the local smoothness constants may change drastically in terms of the absolute difference, but their relative difference can be safely controlled if a clipped update is adopted. Next, we can proceed to the convergence analysis of \texttt{CGTVR-stag} method. For succinctness, we omit the analysis of \texttt{CGTVR-sync} as it is similar but simpler than \texttt{CGTVR-stag}.

\section{The convergence analysis of \texttt{CGTVR-stag}}
\label{section:cvg-analysis}
In this section, we provide the convergence analysis of the \texttt{CGTVR-stag} method based on the assumptions and preliminary results in Sections \ref{section:RUC-regularity} and \ref{section:CGTVR}. We start with a standard iterative descent result based on the local estimation of the smoothness constants that are RUC-regular. 

\begin{lemma}
\label{lemma:stag-descent}
Suppose that the Assumption \ref{assumption:RMSS} - \ref{assumption:mix-spectral} hold and the parameters $\delta,d$ and $r_0$ are selected according to Proposition \ref{proposition:stager-dist-bound}, and $\theta>2$. Then for any $t\geq0$, the iterates of \texttt{CGTVR-stag} satisfy
\begin{equation}
\label{lm:stag-descent-main}
f(\bbxte) \leq f(\bbxt)  -  \frac{(\theta-2)L_{\min}^t}{m} \left\|\ctdb(\byt)\right\|_F^2 +  \frac{L_{\min}^t}{2m}\|\mathbf{1}\bar{\mathbf{x}}^t-\mathbf{x}^t\|^2_F + \frac{\|\bone\bbyt - \byt\|^2_F}{2mL_{\min}^t}+ \frac{\|\omega^t-(\bbyt)^\top\|^2}{2 L_{\min}^t}.
\end{equation}
where $L^t_{\min}=\min_i{L_i^t}$, and $\omega^t:=\frac{1}{m}\sum_{i=1}^m\nabla f_i(\xit)$.
\end{lemma} 

The proof of this lemma is presented in Appendix \ref{appdx:lemma:stag-descent} for succinctness. We should notice that the term $\|\mathbf{1}\bar{\mathbf{x}}^t-\mathbf{x}^t\|^2_F = \sum_i\|x_i^t-\frac{1}{m}\sum_jx_j^t\|^2$ reflect the consensus error, and is expected to go to $0$ due to the mixing property of the matrix $\bW$. A similar consensus behavior should be expected for  $\|\mathbf{1}\bar{\mathbf{y}}^t-\mathbf{y}^t\|^2_F$, and will proved later.  To connect the stochastic descent term (with random errors) $\|\ctdb(\byt)\|_F^2$ with the target quantity $\|\cT_{\beta_{\min}^t}^\delta(\nabla f(\bbxt))\|^2$ that we aim to bound, we need to further provide the following bound, with proof stated in Appendix \ref{appdx:lemma:stag-descent-gradmap}.

\begin{lemma}
\label{lemma:stag-descent-gradmap}
Under the same conditions of Lemma \ref{lemma:stag-descent}, it also holds that 
\begin{eqnarray} 
\label{lm:stag-descent-gradmap-main}
&& f(\bbxte)  \leq  f(\bbxt)  -  \frac{(\theta-2)L_{\min}^t}{8\gamma^2} \big\|\cT_{\beta_{\min}^t}^\delta(\nabla f(\bbxt))\big\|^2  -  \frac{(\theta-2)L_{\min}^t}{2m}\left\|\ctdb(\byt)\right\|_F^2\\
&& \qquad\qquad+ \frac{9L_{\min}^t}{16m}\|\mathbf{1}\bar{\mathbf{x}}^t-\mathbf{x}^t\|^2_F 
 + \frac{9}{16mL_{\min}^t}\|\bone\bbyt - \byt\|^2_F + \frac{9}{16L_{\min}^t}\|\omega^t-(\bbyt)^\top\|^2. \nonumber
\end{eqnarray} 
where $\beta_{\min}^{t}=\theta L_{\min}^t$, and $\gamma = \left\{2,\frac{1+\eta}{2\eta}\right\}$ is the constant selected in \eqref{eqn:ruc-r0} to utilize the RUC property.
\end{lemma} 
Both of the two descent inequalities \eqref{lm:stag-descent-main} and \eqref{lm:stag-descent-gradmap-main} will be useful in the latter analysis. Next, we present the bound on the last term of \eqref{lm:stag-descent-gradmap-main}. 

\begin{lemma}
\label{lemma:bound variance}
Suppose \texttt{CGTVR-stag} takes a constant batch size $|\cB_i^t|\equiv B_i, i\in[m]$, then we have 
\begin{equation}
\label{lm:bound variance-main-1}
\mathbb{E}\left[\left\|\omega^t-(\bbyt)^\top\right\|^2/L_{\min}^t\right]\leq \mathbb{E}\bigg[\frac{\gamma}{m}\sum_{i=1}^m \sum_{r=l(i,t)}^{t-1}\frac{L_{\min}^{r}}{B_i}\left\|x_i^{r+1}-x_i^{r}\right\|^2\bigg].
\end{equation}
\end{lemma}
The proof of this lemma is straightforward and is largely omitted. One only needs to slightly modify the analysis of \cite[Lemma 1, Eq.(39)-(47)]{sun2020improving} while using the local smoothness constants to get 
$$\EE\left[\|\omega^t-(\bbyt)^\top\|^2/L_{\min}^t\right] \leq \mathbb{E}\bigg[\frac{1}{m}\sum_{i=1}^m \sum_{r=l(i,t)}^{t-1}\frac{\bL([x_i^r,x_i^{r+1}])^2}{B_iL_{\min}^t}\left\|x_i^{r+1}-x_i^{r}\right\|^2\bigg],$$
where the detailed computation is omitted for succinctness. Due to Remark \ref{remark:Stag-point-membership}, the line segment $[x_i^r,x_i^{r+1}]\subseteq\cX_j^{l(j,r)}$ for any $i,j$. Hence we can replace $\bL([x_i^r,x_i^{r+1}])$ with $L_{\min}^r$. Now let us fix an arbitrary $i$ and some index $r\geq l(i,t)$ in the above summation, then using Proposition \ref{proposition:stager-dist-bound}, we further know that 
$\frac{L_{\min}^r}{L_{\min}^t}\leq \frac{L_{i}^r}{L_{\min}^t} = \frac{L_{i}^t}{L_{\min}^t} \leq \gamma.$ Hence we complete the proof of this lemma. Next, we present a handy result on the summation of error bounds in Lemma \ref{lemma:bound variance} under the parameter selection $B_i = \tau_i$.
\begin{align}
\label{eqn:summation}
&\sum_{t=0}^{T-1}\mathbb{E}\bigg[\sum_{i=1}^m \sum_{r=l(i,t)}^{t-1}\frac{L_{\min}^{r}}{B_i}\left\|x_i^{r+1}-x_i^{r}\right\|^2\bigg]   \overset{(i)}{\leq}  \sum_{t=0}^{T-1}\mathbb{E}\bigg[\sum_{i=1}^m  \frac{L_{\min}^{t}\tau_i}{B_i}\left\|x_i^{t+1}-x_i^{t}\right\|^2\bigg]\nonumber \\
& \quad\qquad\qquad\qquad\qquad\qquad\overset{(ii)}{=}  \mathbb{E}\bigg[\sum_{t=0}^{T-1}  L_{\min}^{t}\left\|\bxte-\bxt\right\|^2_F\bigg]\\
& \quad\qquad\qquad\qquad\qquad\qquad\overset{(iii)}{\leq}  \mathbb{E}\bigg[\sum_{t=0}^{T-1}  L_{\min}^{t}\Big((1+4)\|\bxte-\bxt\|_F^2 + (1+1/4)\|(\bW-I)(\bxt-\bone\bbxt)\|_F^2\Big)\bigg]\nonumber\\
&\quad\qquad\qquad\qquad\qquad\qquad\overset{(iv)}{\leq} \mathbb{E}\bigg[\sum_{t=0}^{T-1}  5L_{\min}^{t}\Big(\left\|\mathbf{1} \bar{\mathbf{x}}^t-\mathbf{x}^t\right\|^2_F+\left\|\ctdb\left(\byt\right)\right\|^2_F\Big)\bigg].    \nonumber
\end{align} 
In the above argument, (i) is because no matter how the epochs are partitioned for the iteration steps $\{0,1,\cdots,T-1\}$, each $\|\xite-\xit\|^2$ will be summed for at most $\tau_i$ times, (ii) is because we choose $B_i=\tau_i$, (iii) is due to the update rule \eqref{eqn:CGT}, and (iv) is due to Assumption \ref{assumption:mix-spectral}. 

Next, combining Lemmas \ref{lemma:stag-descent}, \ref{lemma:stag-descent-gradmap}, \ref{lemma:bound variance}, and the standard mixing results on the consensus errors $\|\bone\bbxt-\bxt\|^2_F$ and $\|\bone\bbyt-\byt\|^2_F$, we may consider the potential function 
\[
P_t := \mathbb{E}\Big[f(\bar{\mathbf{x}}^{t}) + \alpha_1 L_{\min}^t\big\|\mathbf{x}^t - \mathbf{1} \bar{\mathbf{x}}^t\big\|_F^2 + {\alpha_2}\big\|\mathbf{y}^t - \mathbf{1} \bar{\mathbf{y}}^t\big\|_F^2/{L_{\min}^t}\Big],
\]
for which the following lemma is satisfied by \texttt{CGTVR-stag}, with proof presented in Appendix \ref{appdx:lemma:potential}.
\begin{lemma}
\label{lemma: potential descent 2}
Under the same Assumptions of Lemma \ref{lemma:stag-descent}, and we select the batch size and epoch length s.t. $B_i = \tau_i$ for all $i\in[m]$. Then there is an appropriate set of constants $\alpha_1, \alpha_2, \theta_0>0$ s.t. 
\begin{align}
\sum_{t=0}^{T-1}\!\EE\bigg[\!\frac{(\theta\!-\!2)L_{\min}^t}{32}\big\|\cT_{\!\beta_{\min}^t}^\delta\!\!(\nabla f(\bbxt))\big\|^2 \!+\! \frac{\theta L_{\min}^t}{4m}\big\|\ctdb\!(\byt)\big\|_F^2 \!+\! \frac{3L_{\min}^t}{4m}\|\bxt\!-\!\bone\bbxt\|_F^2\!+\!\frac{\|\byt\!-\!\bone\bbyt\|_F^2}{16mL_{\min}^t}\bigg] \!\leq P_0\!-\!P_T\nonumber.
\end{align}
for arbitrarily selected parameter $\theta\geq\theta_0$. 
\end{lemma}
\noindent In the above lemma, the detailed values of $\alpha_1,\alpha_2$ and $\theta_0$ are presented in Appendix \ref{appdx:lemma:potential} for succinctness. We should notice that $P_0-P_T$ is upper bounded by a constant in that   
\begin{align*}
&P_0 = f(\bar{\bf x}^0) + \alpha_1L_{\min}^0\|{\bf x}^0-\bone\bar{\bf x}^0\|^2_F + \frac{\alpha_2\|{\bf y}^0-\bone\bar{\bf y}^0\|^2_F}{L_{\min}^0} = f(x^0) + \frac{5\sum_{i=1}^m\|\nabla f(x^0)-\nabla f_i(x^0)\|^2}{2m(1-\eta)L_{\min}^0},\\
&P_T = \mathbb{E}\bigg[f(\bar{\bf x}^T) + \alpha_1L_{\min}^T\|{\bf x}^0-\bone\bar{\bf x}^T\|^2_F + \frac{\alpha_2\|{\bf y}^T-\bone\bar{\bf y}^T\|^2_F}{L_{\min}^T}\bigg] \geq \mathbb{E}\big[f(x^T)\big]\geq f^*, 
\end{align*}  
where $f^*$ is defined by Assumption~\ref{ass:existence-minimizer}. This further implies the upper bound on the potential difference for all $T\geq0$ that
\begin{equation}
\label{defn:Delta_f}
P_0-P_T\leq f(x^0) - f^* + \frac{5\sum_{i=1}^m\|\nabla f(x^0)-\nabla f_i(x^0)\|^2}{2m(1-\eta)L_{\min}^0}=:\Delta_f<+\infty.
\end{equation}

\noindent Therefore, combining this upper bound with Lemma \ref{lemma: potential descent 2}, we know that $\EE[\|\bxt-\bone\bbxt\|_F^2]\leq \cO(1/T)$, which indicates that the agents can achieve consensus in a sublinear rate. However, due to the presence of clipping operator, it is not clear what 
$\EE\big[\|\cT_{\beta_{\min}^t}^\delta(\nabla f(\bbxt))\|^2\big]=\cO(1/T)$ stands for. Therefore, we present the following observation of this term. 
\begin{proposition}\label{proposition:clip-measure}
For any $\beta,\delta,\epsilon>0$, and any vector $u\in\RR^d$, we have \vspace{0.1cm}\\
\emph{\bf (i).} Suppose $\epsilon<\delta^2$, then $\|\cT_\beta^\delta(u)\|^2\leq \epsilon$ implies $u/\beta = \cT_\beta^\delta(u)$ and $\|u/\beta\|^2 \leq \epsilon.$ \vspace{0.1cm}\\
\emph{\bf (ii).} Suppose $\epsilon<\delta^2/2$, and $u$ is a random vector, then the bound $\EE\big[\|\cT_\beta^\delta(u)\|^2\big]\leq \epsilon$ implies that the event  $\mathcal{E}:=\{\omega:u/\beta = \cT_\beta^\delta(u)\}$ satisfies 
$${\rm Prob}(\mathcal{E}) \geq  1-\epsilon/\delta^2 = 1-\cO(\epsilon) \qquad\mbox{and}\qquad \EE\big[\|u/\beta\|^2\mid\mathcal{E}\big]\leq \epsilon+ 2\epsilon^2/\delta^2 = O(\epsilon).$$
\end{proposition}
\begin{proof}
The proof of (i) is straightforward because the definition of clipping operator \ref{defn:clip-operator} directly states that  $u/\beta\neq\cT_\beta^\delta(u)$ implies $\|\cT_\beta^\delta(u)\|^2=\delta^2$, which is impossible when $\|\cT_\beta^\delta(u)\|^2\leq\epsilon<\delta^2$. To prove (ii), denote probability ${\rm Prob}(\mathcal{E}^c) = p$, then it suffices to notice that 
\begin{align*}
    \epsilon &\geq \EE\big[\|\cT_\beta^\delta(u)\|^2\big]\\
    & = (1-p)\EE\big[\|\cT_\beta^\delta(u)\|^2\mid\mathcal{E}\big] + p\EE\big[\|\cT_\beta^\delta(u)\|^2\mid\mathcal{E}^c\big] \\
    & =  (1-p)\EE\big[\|\cT_\beta^\delta(u)\|^2\mid\mathcal{E}\big] + p\delta^2,
\end{align*}
which immediately gives 
$$p\leq \epsilon/\delta^2\qquad\mbox{and}\qquad\EE[\|\cT_\beta^\delta(u)\|^2|\mathcal{E}]\leq \frac{\epsilon}{1-\epsilon/\delta^2}\leq \epsilon(1+2\epsilon/\delta^2),$$ where we use the fact that $\frac{1}{1-x}\leq 1+2x$ when $x\in[0,1/2]$, which completes the proof. 
\end{proof}

As a consequence of the above proposition, we know that the behavior of $\cT_\beta^\delta(u)$ and $u/\beta$ are almost identical for small enough $\epsilon$. And we can define the following gradient mapping to measure the near-stationarity of the \texttt{CGTVR-stag}algorithm: 
\begin{equation}\label{defn:GradMap}
\cG(\bbxt) = \beta_{\min}^t\cdot\cT_{\beta_{\min}^t}^\delta(\nabla f(\bbxt))
\end{equation}
where we multiply the inverse stepsize $\beta_{\min}^t$ to cancel out the scaling in the clipping operator $\cT_{\beta_{\min}^t}^\delta(\cdot)$. As a consequence, we obtain the following theorem. 

\begin{theorem}\label{theorem:Iter-Complexity-Bounded-L}
Under the same Assumptions of Lemma \ref{lemma:stag-descent} and we select the parameters according to Lemma \ref{lemma: potential descent 2}. In addition, suppose there is an upper bound $\bar{L}>0$ such that $L_{\min}^t\leq\bar{L}, \forall t\geq0$, then  
\begin{equation}
    \label{thm:Iter-Comp-Bounded-L-0}
    \EE\big[\|\cG(\bar{\bf x}^{t_{\rm out}})\|^2\big] \leq \frac{32\theta^2\bar{L}\Delta_f}{(\theta-2)T}\qquad\mbox{and}\qquad\EE\Big[\frac{1}{m}\big\|x^{t_{\rm out}}_i-\frac{1}{m}\sum_{j=1}^mx_j^{t_{\rm out}}\big\|^2\Big]\leq\frac{4\Delta_f}{3T},
\end{equation}
where the time point $t_{\rm out}$ is randomly selected from $\{0,1,\cdots, T-1\}$.
\end{theorem}

Therefore, it takes $T=\mathcal{O}(1/\epsilon)$ iterations to make the expected squared gradient mapping $\epsilon$-small. 
Beside iteration complexity, an important issue is to bound the sample consumption during the algorithm. However, due to the existence of the early stop mechanism (Algorithm \ref{algorithm:ATVRGT}, Line 14-15), a worst scenario could be that the agents restart their local epoch in every iteration and full gradient is always computed. In this case, no improvement can be expected compared to the naive full gradient deterministic variant of \texttt{CGTVR-stag}. Next, we prove that the restart cannot happen too frequently, using which an expected sample complexity can be obtained. 

\begin{theorem}\label{theorem:restart}
Suppose Assumptions \ref{assumption:RMSS} - \ref{assumption:mix-spectral} hold, we set $B_i = \tau_i = \sqrt{n_i}$ for each agent $i\in[m]$, and each agent $i$ executes Algorithm \ref{algorithm:ATVRGT} Line 15 and early restarts the epoch for $K_i$ times among all the iterations,  if we select $\delta$ and $d$ such that  $c\delta\leq d/2$, then 
$$\EE[K_i]\leq \frac{16\Delta_f\tau_i}{\theta d^2} = \cO(\sqrt{n_i})$$
for $\forall i\in[m]$. If in addition that the upper bound $\bar{L}$ introduced in Theorem \ref{theorem:Iter-Complexity-Bounded-L} exists, then combined with \eqref{thm:Iter-Comp-Bounded-L-0}, we can upper bound the expected total sample consumption for finding some point $\bar{\bf x}^{t_{\rm out}}$ such that  $\EE[\|\cG(\bar{\bf x}^{t_{\rm out}})\|^2]\leq \cO(\epsilon)$ by  $\sum_{i=1}^m \cO\big(n_i^{1.5}+n_i^{0.5}\epsilon^{-1}\big)$. 
\end{theorem}
\begin{proof}
Using the fact that $\bbxte-\bbxt = (\bone/m)^\top\ctdb(\byt)$, Assumption \ref{assumption:RMSS}, and  Lemma \ref{lemma: potential descent 2}, we get
\begin{equation}\label{thm:restart-1}
\sum_{t=0}^{T-1}\EE\big[\|\bbxte-\bbxt\|^2\big]\leq \frac{1}{m}\sum_{t=0}^{T-1}\EE\big[\|\ctdb(\byt)\|^2\big]\leq \frac{4\Delta_f}{\theta}.
\end{equation}
If an agent $i$ start an epoch at time $r$ and execute an early restart at time $r+s$ with $s<\tau_i$. Then Line 14 of Algorithm \ref{algorithm:ATVRGT} suggests that $\|x_i^{r+s}-x_i^r\|\geq d$, together with Lemma \ref{lemma:clip-distance2mean}, we obtain 
$$\|\bar{\bf x}^{r+s}-\bar{\bf x}^{r}\|\geq \|x_i^{r+s}-x_i^r\|-\|x_i^{r+s}-\bar{\bf x}^{r+s}\| - \|x_i^r-\bar{\bf x}^r\| \geq d-c\delta\geq \frac{d}{2}.$$
Jensen's inequality for quadratic function implies that
$$\sum_{t=r}^{r+s-1}\|\bbxte-\bbxt\|^2\geq\frac{\|\bar{\bf x}^{r+s}-\bar{\bf x}^{r}\|^2}{s} \geq \frac{d^2}{4\tau_i}.$$ 
That is, whenever an early restart happens at agent $i$, at least a $d^2/4\tau_i$ constant is contributed to the summation \eqref{thm:restart-1}. And therefore,  we have $$\frac{d^2}{4\tau_i}\cdot \EE[K_i] \leq \frac{4\Delta_f}{\theta},$$ which proves the first bound. Observe that the total number of epochs for agent $i$ in $T$ iterations is upper bounded by $K_i+\lceil T/\tau_i\rceil$. That is, at most this number of full local gradient is computed for agent $i$, while all the other steps are using $B_i$-size batches. The total sample consumption of the algorithm can then be upper bounded by 
$$\sum_{i=1}^m\Big((K_i+\lceil T/\tau_i\rceil)\cdot n_i + TB_i\Big) = \sum_{i=1}^m\cO\big( n_i^{1.5} + \sqrt{n_i}T\big) = \sum_{i=1}^m\cO\big(n_i^{1.5}+n_i^{0.5}\epsilon^{-1}\big),$$
where the last step is due to Theorem \ref{theorem:Iter-Complexity-Bounded-L}, which suggests that $T = \cO(\epsilon^{-1})$.
\end{proof}

\section{Extended discussion on local smoothness constants}
\label{section:discussion}
In Section \ref{section:cvg-analysis}, we assume all the local Lipschitz constants $L_{\min}^t$ to be upper bounded by some constant $\bar{L}$ for the ease of presentation, which commonly holds when the objective function is coercive, or more generally, level bounded. However, this simplification leaves two important questions unanswered. First, how does the method differ from the standard non-adaptive gradient tracking method that directly estimates and uses $\bar{L}$ as their ``global'' Lipschitz constant? Second, what if the problem does not have bounded level sets so that the $\bar{L}$ does not exist? In this section, we discuss these questions. To decouple the stochasticity between the correlated random variables $L_{\min}^t$ and $\|\mathcal{G}(\bar{x}^t)\|$, let us combine Lemma \ref{lemma: potential descent 2} and \eqref{defn:GradMap} to obtain that 
\begin{align}
\sum_{t=0}^{+\infty} \EE\bigg[\frac{\theta-2}{32\theta^2 L_{\min}^t}\|\mathcal{G}(\bar{\bx}^t)\|^2 + \frac{\theta L_{\min}^t}{4m}\big\|\ctdb(\byt)\big\|_F^2\bigg] \leq \Delta_f<+\infty\nonumber,
\end{align}
where we let $T\to+\infty$ in Lemma \ref{lemma: potential descent 2} because it holds for any $T\geq1$, while the upper bound \eqref{defn:Delta_f} is always true. Note that all the random variables in the expectations are all non-negative, then the monotone convergence theorem \cite{yeh2006real} implies that we can exchange the summation (infinite) and the expectation. Then applying Markov's inequality to the inequality after exchanging summation and expectation gives 
\begin{align}
\label{eqn:Markov-Bound}
\sum_{t=0}^{+\infty} \left(\frac{\theta-2}{32\theta^2 L_{\min}^t}\|\mathcal{G}(\bar{\bx}^t)\|^2 + \frac{\theta L_{\min}^t}{4m}\big\|\ctdb\!(\byt)\big\|_F^2\right) \leq \frac{\Delta_f}{p}  \quad\text{w.p.}\quad 1-p.
\end{align}

Next, we can proceed our analysis based on this bound.

\subsection{Discussion given finite upper bound $\bar{L}$}
By assuming the existence of a large enough $\bar{L}<+\infty$ so that the $L_{\min}^t$ in 
inequality \eqref{eqn:Markov-Bound} are all upper bounded by $\bar{L}$, we can obtain for any $T\geq1$ that
\begin{equation}
\label{eqn:min-Lbar}
\min_{0\leq t\leq T-1}\|\cG(\bar{\bf x}^{t})\|^2 \leq \bar{L}\cdot \frac{32\theta^2\Delta_f}{p(\theta-2)T}\quad\text{w.p.}\quad 1-p,
\end{equation}
which is a high probability variant of Theorem \ref{theorem:Iter-Complexity-Bounded-L} that utilizes Markov's inequality. First of all, we should admit that, in the worst case, this bound is tight because one can always consider quadratic functions whose local Lipschitz constant $\bL_i(\cdot)\equiv L$ is a constant. In this case, our adaptive selection of $L_i^t$ exhibits no difference from standard constant $L$ approaches. However, in some cases where a significant part of local $L_i^t$ are much smaller than the maximum $\max_t L_i^t$, the adaptive selection of $L_i^t$ can indeed achieve some advantage. In this situation, an alternative sharper bound can be derived from \eqref{eqn:Markov-Bound}:
\begin{equation}
\label{eqn:min-Harmoric}
\min_{0\leq t\leq T-1}\|\cG(\bar{\bf x}^{t})\|^2 \leq \frac{\sum_{t=0}^{T-1} \frac{1}{L_{\min}^t}\|\mathcal{G}(\bar{\bx}^t)\|^2}{\sum_{t=0}^{T-1} \frac{1}{L_{\min}^t}} \leq \frac{T}{\sum_{t=0}^{T-1}\frac{1}{L_{\min}^t}}\cdot  \frac{32\theta^2\Delta_f}{p(\theta-2)T} \quad\text{w.p.}\quad 1-p.
\end{equation}
Note that through harmonic-arithmetic inequality, one must have 
$$\frac{T}{\sum_{t=0}^{T-1}\frac{1}{L_{\min}^t}} \leq \frac{1}{T}\sum_{t=0}^{T-1} L_{\min}^t \leq \max_{0\leq t\leq T-1} L_{\min}^t\leq \bar{L}$$
and the equality can only be achieved when $L_{\min}^0 = \cdots = L_{\min}^{T-1} = \bar{L}$. That is, an ``effective'' Lipschitz constant in the upper bound is the harmonic average of local Lipschitz constants, which can be much smaller than their common upper bound $\bar{L}$. In fact, in the first experiment on quadratic inverse problem in Section \ref{section:numerical}, where the maximum local Lipschitz constant is much larger than the average, \eqref{eqn:min-Harmoric} can provide a much sharper bound than \eqref{eqn:min-Lbar}, explaining the potential advantage of adaptively estimating $L_i^t$.

\subsection{Discussion without finite upper bound $\bar{L}$}

We should note that although an upper bound $\bar{L}$ is assumed to exist, it is not required a prior in the implementation of the algorithm due to the exploitation of local smoothness constants. When the problem is coercive, the existence of such an $\bar{L}$ is often satisfied by practical applications. Yet, theoretically, there indeed exist corner cases where such an $\bar{L}$ does not exists. This can only happen in  case where the iteration sequence, although with monotonically decreasing potential function $P_t$, diverges to infinity. For example, one may consider the following hand-crafted two-dimensional instance 
$$\min_{x_1,x_2}\,\, f(x_1,x_2):=\frac{1}{\ln (x_1^2+2)} + (\sin^2 x_1^2 - x_2^2)^2.$$
Clearly, a global optimal solution does not exist as the infimum $0$ is only approached (but not obtained) when $|x_1|\to+\infty$. In the meantime, a direct computation shows that $\|\nabla^2f(x_1,x_2)\|\to+\infty$, indicating an unbounded local smoothness constant.

In this case, we cannot simply argue the existence of a finite $\bar{L}<+\infty$ to upper bound all the locally estimated Lipschitz constants. However, akin to the analysis in \cite[Section 4]{zhang2025stochastic}, it is still possible to bound the local Lipschitz constants before finding an  $\epsilon$-solution, as stated in the next proposition.  
\begin{proposition}
\label{thm:bounded-region}
Under the setting of Lemma \ref{lemma:stag-descent} while taking $\theta\geq4$, and we select the parameters according to Lemma \ref{lemma: potential descent 2}. Let $T_\epsilon := \inf\{ t : \|\mathcal{G}(\bbxt)\|^2 \le \epsilon\}$
be the random time point when the algorithm first hits an $\epsilon$-small squared gradient mapping. Then w.p. $1-p$, it holds for all $t<T_\epsilon$ and $i\in[m]$ that 
$$\big\|x_i^t-\bar{\bx}^{0}\big\|
\le
\frac{8\,\Delta_f}{\sqrt{\epsilon}\cdot p}+\frac{\sqrt{2m}\delta}{1-\eta}\qquad\mbox{and}\qquad L_i^t
\leq  
\bL_i\left(B\left(\bar{\bx}^0,\frac{8\Delta_f}{\sqrt{\epsilon}\cdot p} + \frac{\sqrt{2m}\delta}{1-\eta}+r_0\right)\right).$$ 
\end{proposition}

\begin{proof}

Applying the relationship between 
$\|\bbxte-\bbxt\|^2$ and $\|T^\delta_{\beta_t}(\byt)\|^2$ to \eqref{eqn:Markov-Bound} yields w.p. $1-p$ that
\begin{eqnarray*}
    \frac{\Delta_f}{p} & \geq & 
\sum_{t=0}^{+\infty} \bigg(
\frac{\theta-2}{32\theta^2 L_{\min}^t}\|\mathcal{G}(\bbxt)\|^2
+
\frac{\theta L_{\min}^t}{4}\|\bbxt-\bbxte\|^2\bigg)\\
& \geq & \sum_{t=0}^{+\infty}2\sqrt{\frac{\theta-2}{32\theta^2 L_{\min}^t}\cdot\frac{\theta L_{\min}^t}{4}\cdot\|\mathcal{G}(\bbxt)\|^2 \cdot \|\bbxt-\bbxte\|^2} \\
& = & \sum_{t=0}^{+\infty}\sqrt{\frac{\theta-2}{32\theta}} \|\mathcal{G}(\bbxt)\| \cdot \|\bbxt-\bbxte\| \\
&\geq& \frac{1}{8}  
\sum_{t=0}^{+\infty}
\|\mathcal{G}(\bbxt)\| \cdot \|\bbxt-\bbxte\|,
\end{eqnarray*} 
where the last line is due to $\theta\geq4$. Then truncating the infinite summation at $T_\epsilon$ yields
\[
\frac{8\Delta_f}{p} 
\ge 
\sum_{t=0}^{T_\epsilon-1}
\|\mathcal{G}(\bbxt)\|\cdot \|\bbxt-\bbxte\|
\ge  
\sum_{t=0}^{T_\epsilon-1}
\sqrt{\epsilon}\, \|\bxt-\bbxte\|  \quad \mbox{ w.p. }\quad 1-p.
\]
Therefore, we have with probability at least $1-p$ that 
\begin{align}
\max_{k<T_\epsilon} \max_{i\in[m]} \|x_i^k - \bar{\bx}^0\| & \leq  \max_{k<T_\epsilon}\max_{i\in[m]}\Big(\|\bar{\bx}^k - \bar{\bx}^0\| + \|\bar{\bx}^k - x_i^k\|\Big)\nonumber\\
& \leq \sum_{t=0}^{T_\epsilon-1}
\|\bar{\bx}^t-\bar{\bx}^{t+1}\| + \frac{\sqrt{2m}\delta}{1-\eta} \leq 
\frac{8\Delta_f}{\sqrt{\epsilon}\cdot p} + \frac{\sqrt{2m}\delta}{1-\eta}. \nonumber
\end{align} 
where we use the fact that the clipped update ensures that each local iterate to lie within  $\frac{\sqrt{2m}\delta}{1-\eta}$-distance of the averaged iterate. Therefore, for all pairs $(i,t)$ at which the local Lipschitz  constant $L_i^t = \bL_i(\cX_i^t)$ are estimated, we know 
$$\cX_i^t = B(x_i^t,r_0)\subseteq B\left(\bar{\bx}^0,\frac{8\Delta_f}{\sqrt{\epsilon}\cdot p} + \frac{\sqrt{2m}\delta}{1-\eta}+r_0\right).$$
Then using the monotonicity of the set function $\bL_i$ proves the proposition. 
\end{proof}

Note that the upper bound on $L_i^t$ only relies on the target accuracy $\epsilon$, the confidence level $p$, and predetermined algorithmic parameters. It shows that, even without a finite $\bar{L}$ for all $t\geq0$, one can still bound $L^t_i$ for $t<T_\epsilon$ with high probability, and hence provide a high probability bound for $T_\epsilon$.  Define $$R_{\epsilon,p}:=\frac{8\Delta_f}{\sqrt{\epsilon}\cdot p} + \frac{\sqrt{2m}\delta}{1-\eta}+r_0 = \mathcal{O}(\epsilon^{-1/2}p^{-1})$$
as the shorthand for the radius specified in Proposition \ref{thm:bounded-region}, we provide the following remark. 
\begin{remark}\label{remark:poly}
    Under the setting of Proposition \ref{thm:bounded-region}, and suppose in addition that the local smoothness constants exhibit a power growth in the sense that  $\mathbb{L}_i(B(\bar{\bf x}^0,R)) = \mathcal{O}(R^Q)$ for some $Q\geq0$. Then, combining this proposition and the analysis of Theorem \ref{theorem:Iter-Complexity-Bounded-L} implies that, with probability at least $1-p$, we have $T_\epsilon\leq\mathcal{O}(R^Q_{\epsilon,p}/\epsilon) = \mathcal{O}(\epsilon^{-(Q/2+1)}p^{-Q})$, matching the $\mathcal{O}(\epsilon^{-(Q/2+1)})$ complexity in \cite{zhang2024first} for deterministic and centralized problems.
\end{remark} 
Therefore, the worst-case iteration complexity of the method can depend on the growth rates of $\mathbb{L}_i$, which, to some degree, represent the level of non-Lipschitz smoothness. For general growth function that does not have an explicit form, one may replace the $R_{\epsilon,p}^Q$ factor in the above iteration complexity bound by $\max_i\{\mathbb{L}_i(B(\bar{\bf x}^0,R_{\epsilon,p}))\}$. Nevertheless, this additional $\epsilon$-dependency is inevitable for problems without globally Lipschitz gradients, see discussion in \cite{zhang2025stochastic}.

\section{Numerical experiments}  
\label{section:numerical}

In this section, we conduct numerical experiments on quadratic inverse and dimension reduction problems to evaluate \texttt{CGTVR-stag} and \texttt{CGTVR-sync}, with benchmarks being the established decentralized stochastic variance reduction methods, including \texttt{GT-VR} \cite{jiang2022distributed}, \texttt{GT-SARAH} \cite{xin2022fast}, and \texttt{D-GET} \cite{sun2020improving}.  \vspace{0.2cm}

\noindent{\bf Quadratic inverse problem. }  
First, we consider the distributed quadratic inverse problem studied in \cite{zhao2018distributed}, which is formulated as 
\begin{equation}
\label{eq:quardratic inverse}
\mathop{\rm minimize}_{{x}_1,\cdots, {x}_m \in\RR^d}\,\, \frac{1}{m}\sum_{i=1}^{m}\frac{1}{n_i}\sum_{j=1}^{n_i} \big(b_{i,j}-\langle a_{i,j},x_i\rangle^2\big)^2,  \qquad \st \quad x_i = x_j, \forall\, i,j\in[m],
\end{equation}
where $a_{i,j}\in\RR^d$ are sampling vectors and $b_{i,j}=\langle a_{i,j}, {x}_{\rm true}\rangle^2+\epsilon_{i,j}$ are measurement data  with ${x}_{\rm true}$ being the signal vector and $\epsilon_{i,j}$ being i.i.d. Gaussian noises.  We consider two 32 × 32 image signals, Barbara and Baboon, reshaped and then normalized as  1024-dimensional vectors, named $x_{\text{true}}$. We use Gaussian sampling scheme \cite{candes2015phase} where each $a_{i,j}\sim \mathcal{N}(0,I_d)$ is generated by standard Gaussian distribution and the additive Gaussian noise is generated by $\epsilon_{i,j}\sim \mathcal{N}(0,0.05)$. Each agent possesses a local dataset of $n_1 =\cdots=n_m=5000$ sampling vectors and the corresponding measurement data. We will denote this number as $n$ for simplicity. For this formulation, 
direct estimation of the local smoothness constant for each agent $i$ should be 
$$\bL_i(\cX) = \max\bigg\{\Big\|\frac{1}{n}\sum_{j=1}^n b_{i,j}a_{i,j}a_{i,j}^\top\Big\| + \frac{3}{n}\sum_{j=1}^n\|a_{i,j}\|^4\cdot\|x\|^2:x\in\cX\bigg\},$$
see e.g. \cite{bolte2018first}, which easily explodes to $10^{10}$ due to its dimension dependence. This estimation can be too conservative for practical computation. In fact, using the property of Gaussian distribution, with high probability, an approximated estimation of local smoothness constant will be around
$$\bL_i(\cX) \approx \max\bigg\{\Big\|\frac{1}{n}\sum_{j=1}^n b_{i,j}a_{i,j}a_{i,j}^\top\Big\| + 5\|x\|^2: x\in\cX\bigg\},$$
which corresponds to case (3) of Proposition \ref{proposition:growth}. In this case, according to Appendix \ref{appdx:proposition:growth}, $r_0$ can in fact be selected adaptively as $r_0\propto\|x^{\rm c}_i\|$ where $x^{\rm c}_i$ is the starting point of each local epoch, and then by Theorem \ref{theorem:Iter-Complexity-Bounded-L} and \ref{theorem:restart}, we can select $d = \frac{2r_0}{3}$ and $\delta = \frac{1r_0}{3(c+1)}$ so that $d+(c+1)\delta\le r_0$ and $c\delta<d$ for \texttt{CGTVR-stag} and select $\delta=r_0/(c+\tau)$ for \texttt{CGTVR-sync} so that \eqref{eqn:sync-epoch-radii} is satisfied. Here $c = \frac{2\sqrt{2m}}{1-\eta}$ and $\eta$ can be computed from the mixing matrix that depends on the input network, where we consider $m=16$ agents with three different network topologies, a ring network, a 4 × 4 grid network, and an Erdős–Rényi random graph $G(16,0.2)$, see \cite{renyi1959random}.

\vspace{0.25cm}

For \texttt{GT-VR} that uses a loopless variance reduction scheme associated with a restart probability $p$, the value $p>0.5$ suggested in \cite[Corollary IV.1]{jiang2022distributed} does not behave well. Such a value of $p$ suggests that a restart will happen every 2 steps on average, where a full batch needs to be taken. This makes algorithm essentially ``deterministic'' because the full batch is too frequently taken and the numerical performance will be very bad. In our experiment, we do not implement the default batch size and $p$ in \cite{jiang2022distributed}. Instead, we set $p= n^{-1/3}$ and a batch size of $B=n^{2/3}$ in \texttt{GT-VR}. For \texttt{GT-SARAH}, \texttt{D-GET}, \texttt{CGTVR-sync} and \texttt{CGTVR-stag}, all of them take the same batch size $B=\sqrt{n}$ and a (maximal) epoch length $\tau = B$. For the stepsize, \texttt{GT-VR}, \texttt{GT-SARAH}, and \texttt{D-GET}  are all tuned from a logarithmic grid $\{10^{-8},10^{-6},\cdots,10^{2}\}$ and we select the best performing one in the report. For \texttt{CGTVR-stag} and \texttt{CGTVR-sync}, we tune $\theta$ from the grid $\{10^{-3},10^{-1},\cdots,10^{3}\}$ and select the best performing one to report.  We report the numerical result in Figure \ref{fig:Baboon} and \ref{fig:Barbara}, where the curves of the squared gradient norm $\|\nabla f(\bar{x}_t)\|^2$ (top row) and objective value (middle row) versus the number of data passes are plotted, where a data pass is defined by 
$$\text{\# data pass}:=\frac{\text{\# of samples consumed}}{\text{full batch size}}.$$ 
We also report the consensus error which measured by $\sqrt{\frac{1}{n}\sum_{i=1}^n \|x_i^t-\bar{x}_t\|^2}$ where $
\bar{x}_t:=\frac{1}{n}\sum_{i=1}^n x_i^t$ stands for the average iterate.

\begin{figure}[H]
    \centering
    \begin{minipage}{0.32\linewidth}
    \centering
    \includegraphics[width=1\linewidth]{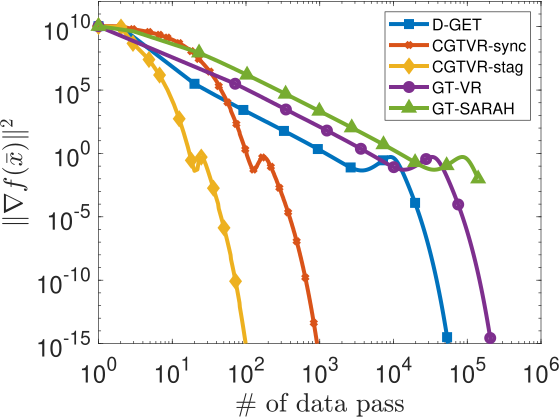} 
    \end{minipage}
    \hspace{0.001\linewidth}
     \begin{minipage}{0.32\linewidth}
    \centering
    \includegraphics[width=1\linewidth]{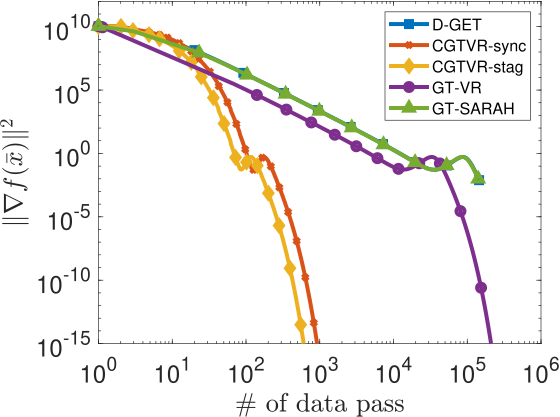} 
    \end{minipage}  
    \hspace{0.001\linewidth}
    \begin{minipage}{0.32\textwidth}
        \centering
        \includegraphics[width=\linewidth]{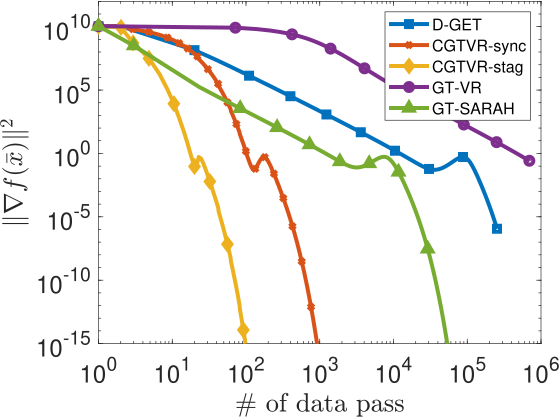} 
    \end{minipage}
    \begin{minipage}{0.32\linewidth}
    \centering
    \includegraphics[width=1\linewidth]{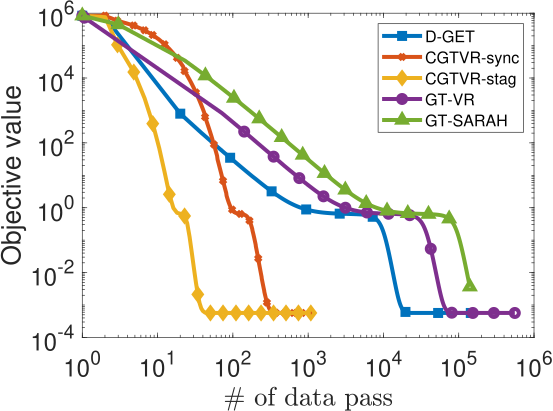}
    \end{minipage}  
    \hspace{0.001\linewidth}
    \begin{minipage}{0.32\linewidth}
    \centering
    \includegraphics[width=1\linewidth]{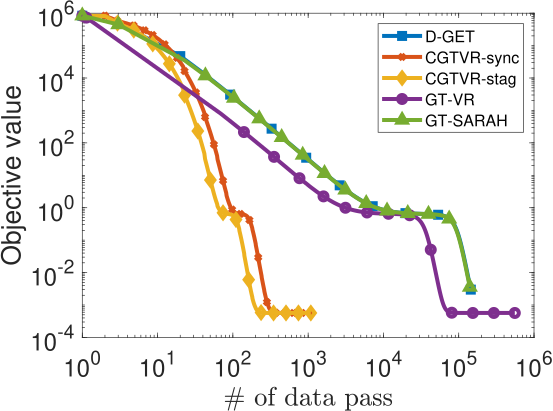}
    \end{minipage}  
    \hspace{0.001\linewidth}
    \begin{minipage}{0.32\linewidth}
    \centering
    \includegraphics[width=1\linewidth]{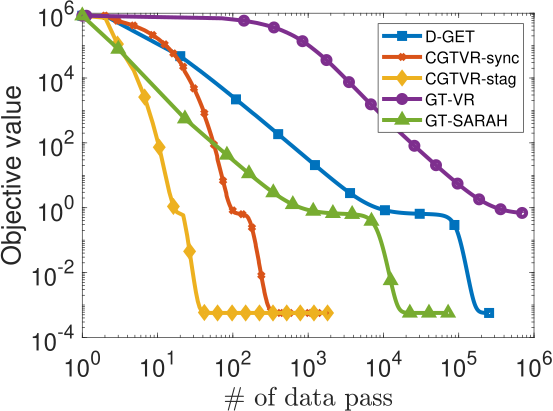}
    \end{minipage}  
    \begin{minipage}{0.32\linewidth}
    \centering
    \includegraphics[width=1\linewidth]{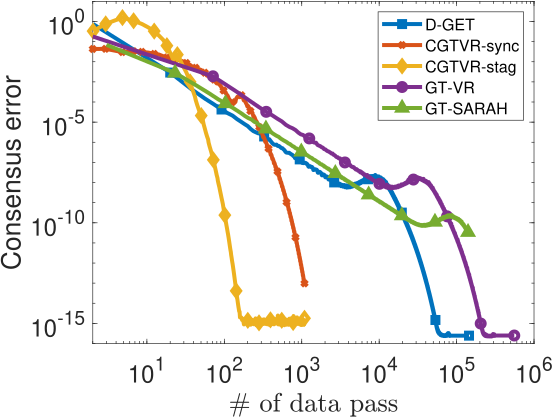}
    \end{minipage}  
    \hspace{0.001\linewidth}
    \begin{minipage}{0.32\linewidth}
    \centering
    \includegraphics[width=1\linewidth]{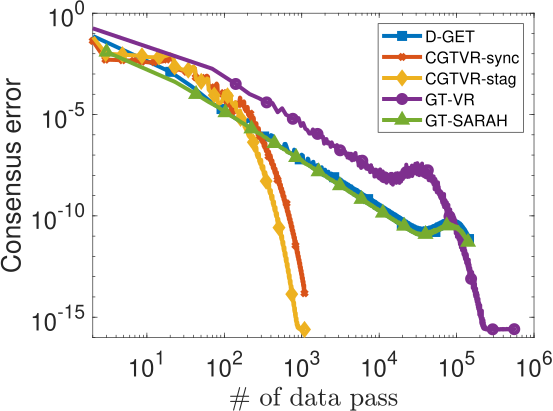}
    \end{minipage}  
    \hspace{0.001\linewidth}
    \begin{minipage}{0.32\linewidth}
    \centering
    \includegraphics[width=1\linewidth]{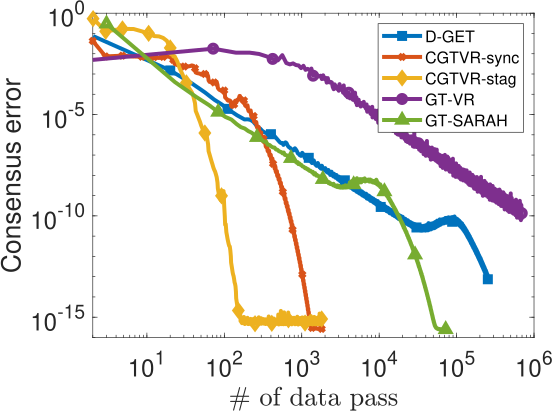}
    \end{minipage}  
    
    \caption{Result of Baboon data. Three columns stand for ring, grid, and random networks. }
    \label{fig:Baboon}
\end{figure}

In Figure~\ref{fig:Baboon} and Figure~\ref{fig:Barbara}, we compare the performance of \texttt{D-GET}, \texttt{CGTVR-sync}, \texttt{CGTVR-stag}, \texttt{GT-VR}, and \texttt{GT-SARAH} under ring, grid, and random network topologies. 

Among the three network topologies, \texttt{CGTVR-stag} and \texttt{CGTVR-sync} consistently demonstrate the most favorable convergence performance, often reducing the gradient to a sufficiently small size within a few hundred data passes. Except for the grid network experiments where \texttt{CGTVR-stag} and \texttt{CGTVR-sync} share similar performance, the staggered scheme outperforms the synchronized scheme in all other experiments. In terms of the compared benchmarks, after resetting $p$ and $B$ instead of using their suggested values in the corresponding paper, \texttt{GT-VR} performs comparable to the other two benchmarks;  \texttt{GT-SARAH} and \texttt{D-GET} share similar behavior in most cases, but \texttt{D-GET} is faster in some cases. Note that in this example, the local smoothness constant is significant better around the optimal solution than that of a Gaussian random initial point, \texttt{CGTVR-stag} can be much faster than the compared non-adaptive benchmarks so that we must use a log-log scale plot. This observation validates our theoretical comparison between \eqref{eqn:min-Lbar} and \eqref{eqn:min-Harmoric} in Section \ref{section:discussion}. To plot the point of initial solution with 0 data pass, we modify the x-axis by adding 1 data pass to every point on the plotted curves. Finally, for the consensus error curves, because all agents share the same consensus initialization with 0 consensus error, the first point (0,0) in these curves is not representable on a log--log axis and is therefore omitted. Consequently, the consensus-error curves start from the first nonzero entry. The performance of all compared algorithms in terms of the consensus error is similar to the gradient and function value curves, and the comments are thus omitted.

\vspace{0.5cm}
\noindent{\bf Dimension reduction problem. } Let \( {x}_{i,k} \in \mathbb{R}^d, \) denote the $k$-th local data point at agent $i$. We consider the following dimension reduction task:
\begin{equation}
\min_{U_i\in\RR^{d\times d'}} \sum_{i=1}^m \sum_{k=1}^{n_i} \left\| {x}_{i,k} - U_i U_i^\top {x}_{i,k} \right\|^2, \quad \text{subject to } {U}_i = U_j, \quad \forall\, i,j \in [m],
\end{equation}
where $d'<d$ and we aim to compress the data point $x\in\RR^d$ to a lower dimensional vector $x'=U^\top x\in\RR^{d'}$. Similar to the quadratic inverse problem, the local smoothness constant for each agent $i$ can be estimated by 
\begin{eqnarray}
\bL_i(\cX) &\leq& \max\bigg\{ \sum_{k=1}^{n_i}\|x_{i,k}\|^2\cdot\|UU^\top-I\|+2\|U\|^2_{F} : U\in\cX\bigg\}, \nonumber\\
 & \leq & \sum_{k=1}^{n_i}\|x_{i,k}\|^2+ \Big(\sum_{k=1}^{n_i}\|x_{i,k}\|^2+2\Big)\cdot\max_{U\in\mathcal{X}}\|U\|^2_F, \nonumber
\end{eqnarray}
which also corresponds to case (3) of Proposition \ref{proposition:growth},
and the parameter selections of $r_0, d$, and $\delta$ are also similar. 

We test the algorithms on a public Bank Customer Segmentation dataset from Kaggle, which contains 8,950 records of bank customers with $d=17$ features, and we aim to reduce the dimension to $d'=2$. We partition the dataset randomly among 6 agents with 4000, 250, 250, 250, 250, and 3950 samples, respectively.  We still test the algorithms on ring, 3 × 2 grid, and Erdős–Rényi random graph $G(6,0.2)$ network topologies.

\begin{figure}[H]
    \centering
    \begin{minipage}{0.32\linewidth}
    \centering
\includegraphics[width=1\linewidth]{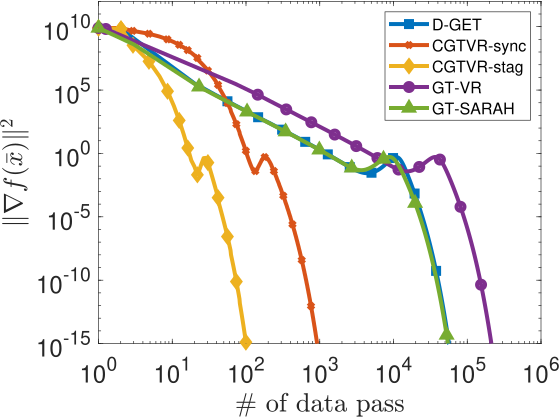} 
    \end{minipage}
    \hspace{0.001\linewidth}
     \begin{minipage}{0.32\linewidth}
    \centering
\includegraphics[width=1\linewidth]{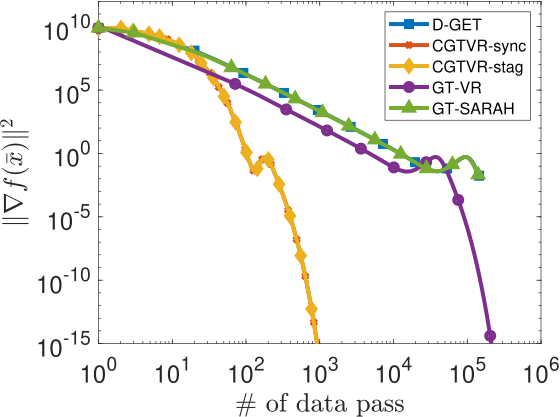} 
    \end{minipage}  
    \hspace{0.001\linewidth}
    \begin{minipage}{0.32\linewidth}
        \centering
\includegraphics[width=\linewidth]{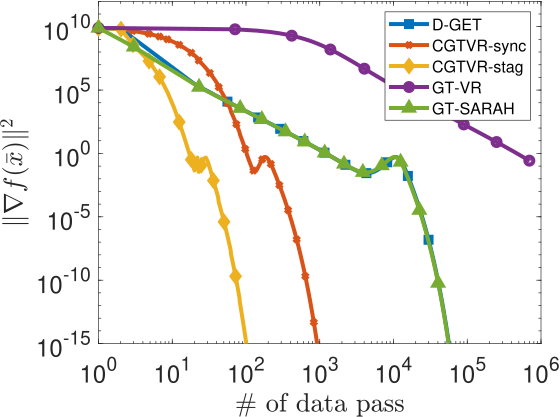} 
    \end{minipage}
    \begin{minipage}{0.32\linewidth}
    \centering
    
\includegraphics[width=1\linewidth]{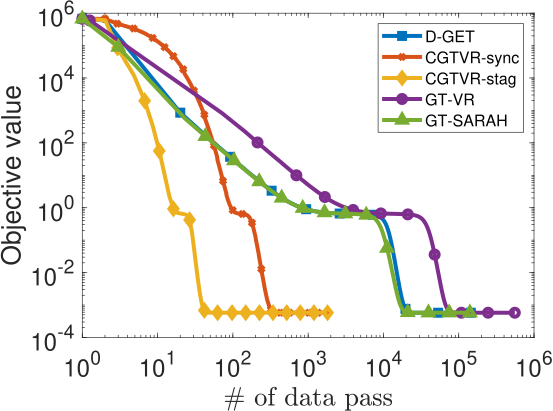}
    \end{minipage}  
    \hspace{0.001\linewidth}
    \begin{minipage}{0.32\linewidth}
    \centering
    \includegraphics[width=1\linewidth]{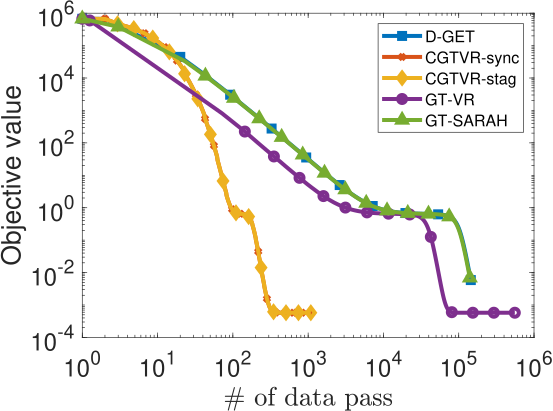}
    \end{minipage}  
    \hspace{0.001\linewidth}
    \begin{minipage}{0.32\linewidth}
    \centering
\includegraphics[width=1\linewidth]{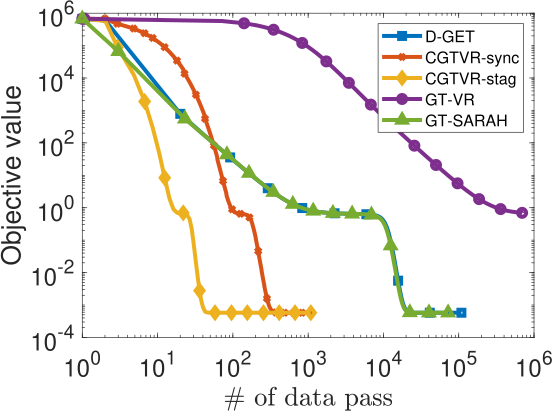}
    \end{minipage}  
    \begin{minipage}{0.32\linewidth}
    \centering
    
\includegraphics[width=1\linewidth]{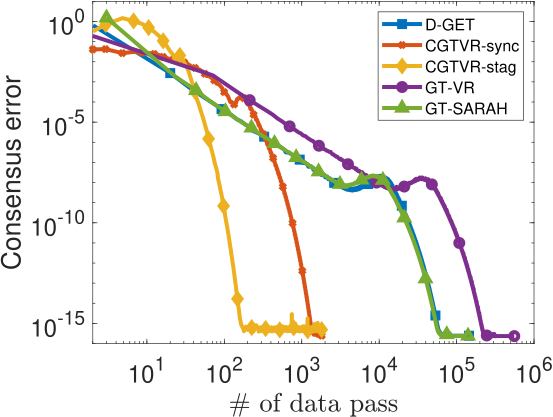}
    \end{minipage}  
    \hspace{0.001\linewidth}
    \begin{minipage}{0.32\linewidth}
    \centering
    \includegraphics[width=1\linewidth]{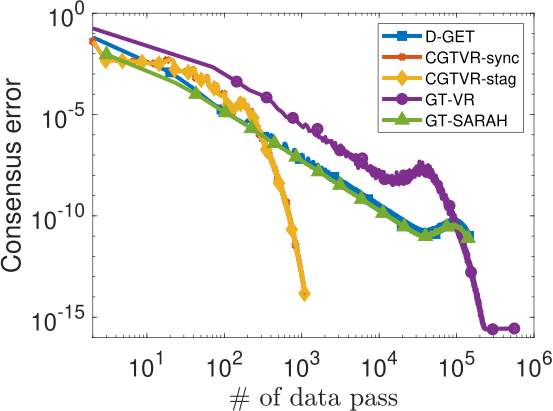}
    \end{minipage}  
    \hspace{0.001\linewidth}
    \begin{minipage}{0.32\linewidth}
    \centering
\includegraphics[width=1\linewidth]{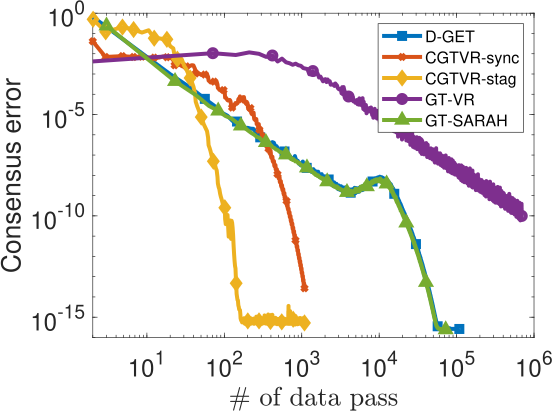}
    \end{minipage}  
     \caption{Result of Barbara data. Three columns stand for ring, grid, and random networks. }
    \label{fig:Barbara}
\end{figure}

The parameter selection is almost the same as that of quadratic inverse problem experiment, and hence we only emphasize the difference on batch size selection. For \texttt{GT-SARAH} and \texttt{D-GET}, the batch size is selected as $B = \max_i\sqrt{n_i}$ and $\tau = B$. For \texttt{GT-VR} and \texttt{CGTVR-stag} that allow staggered restarts at different agents, \texttt{GT-VR} implements a local batch size $B_i = {n_i}^{2/3}$ with local restart probability $p_i = n_i^{-1/3}$, \texttt{CGTVR-stag} implements a local batch size of $B_i = \sqrt{n_i}$ and epoch length of $\tau_i = B_i.$ For \texttt{CGTVR-sync}, the batch size is also selected as $B = \max_i\sqrt{n_i}$ and $\tau = B$.

\begin{figure}[H]
    \centering
    \begin{minipage}{0.32\linewidth}
    \centering    \includegraphics[width=1\linewidth]{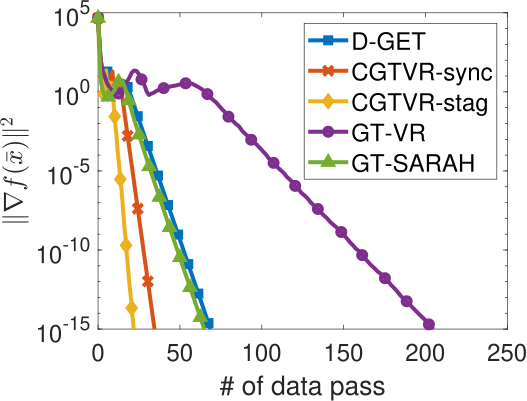} 
    \end{minipage}
    \hspace{0.001\linewidth}
     \begin{minipage}{0.32\linewidth}
    \centering
    \includegraphics[width=1\linewidth]{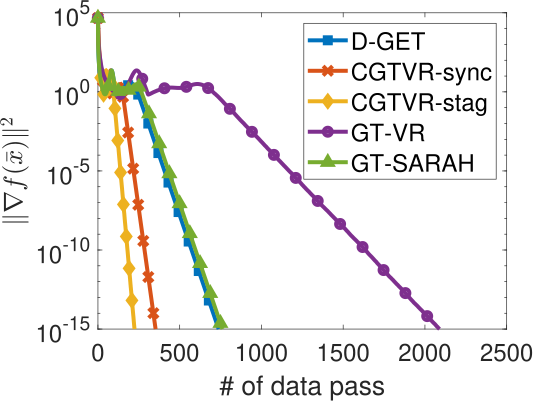} 
    \end{minipage}  
    \hspace{0.001\linewidth}
    \begin{minipage}{0.32\textwidth}
        \centering
        \includegraphics[width=\linewidth]{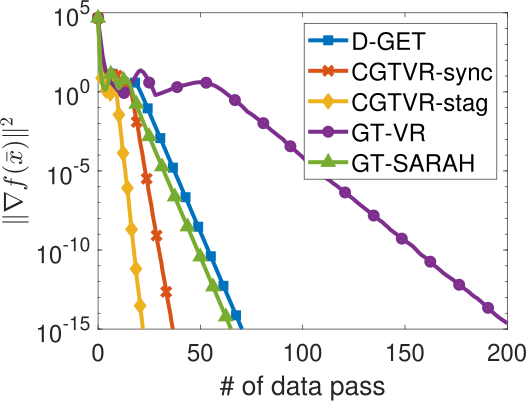} 
    \end{minipage}
    \vskip 0.25cm
    \begin{minipage}{0.32\linewidth}
    \centering
    \includegraphics[width=1\linewidth]{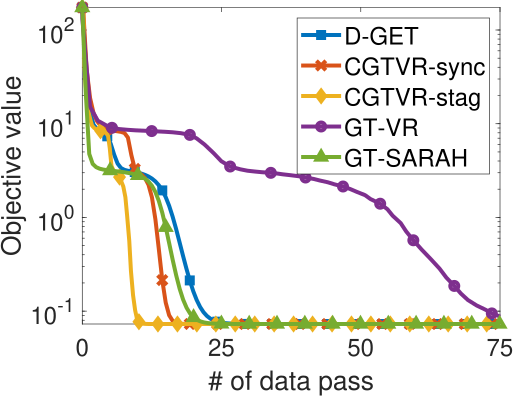}
    \end{minipage}  
    \hspace{0.001\linewidth}
    \begin{minipage}{0.32\linewidth}
    \centering
    \includegraphics[width=1\linewidth]{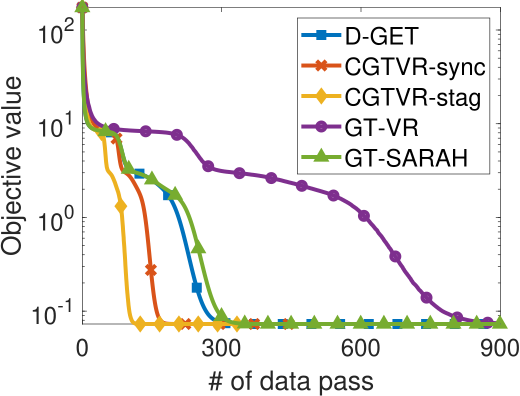}
    \end{minipage}  
    \hspace{0.001\linewidth}
    \begin{minipage}{0.32\linewidth}
    \centering
    \includegraphics[width=1\linewidth]{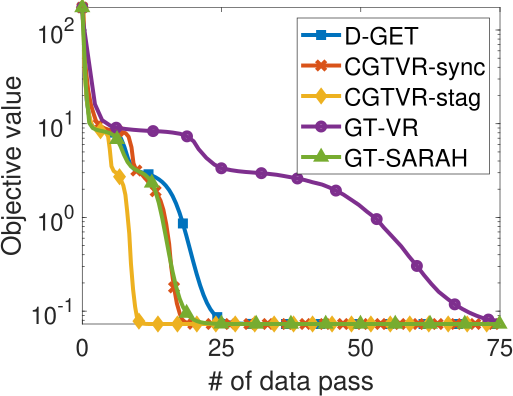}
    \end{minipage}  
    \begin{minipage}{0.32\linewidth}
    \centering
    \includegraphics[width=1\linewidth]{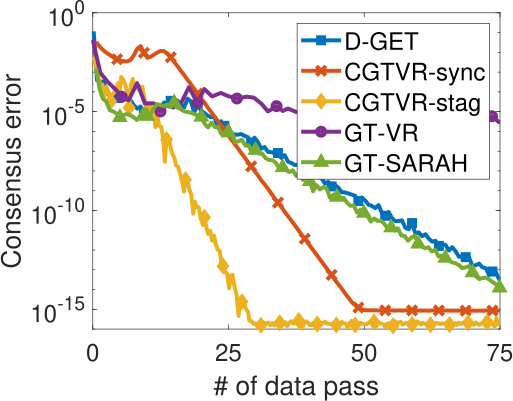}
    \end{minipage}  
    \hspace{0.001\linewidth}
    \begin{minipage}{0.32\linewidth}
    \centering
    \includegraphics[width=1\linewidth]{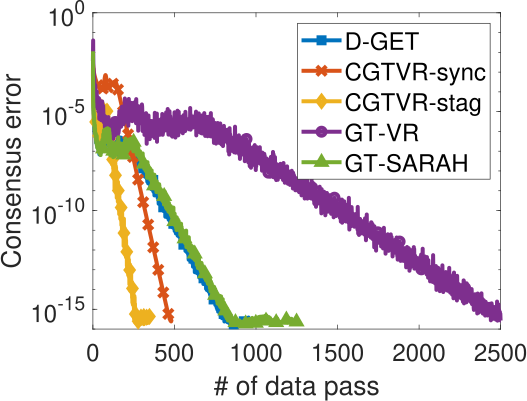}
    \end{minipage}  
    \hspace{0.001\linewidth}
    \begin{minipage}{0.32\linewidth}
    \centering
    \includegraphics[width=1\linewidth]{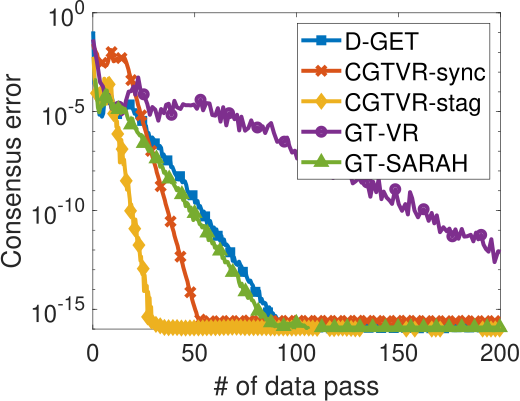}
    \end{minipage}  
     \caption{Bank Customer Segmentation Data}
    \label{fig:bank}
\end{figure}

As illustrated in Figure \ref{fig:bank}, the performance of  the four algorithms is similar to the previous quadratic inverse experiments, where \texttt{CGTVR-stag} and \texttt{CGTVR-sync} consistently outperform the non-adaptive benchmarks. The only difference is that even after resetting $p$ and $B$, the \texttt{GT-VR} algorithm still performs significantly worse than the other two benchmarks. We will not repeat the similar comments again for the sake of succinctness.

\section*{Data Availability} The datasets analyzed during the current study are as follows. 
\begin{itemize}
    \item The simulated datasets were generated using standard test images: \textbf{Baboon} and \textbf{Barbara}. These images are publicly available and widely used in the image processing literature.

    \item The real-world dataset used in this study is publicly available on Kaggle and can be accessed at: \\
\url{https://www.kaggle.com/datasets/sonugupta0932/bank-customer-segmentation}.
\end{itemize}

\section*{Declarations}
\noindent{\bf Conflict of interest} 
The authors have no relevant financial or non-financial interests to disclose.

\newpage
\appendix
\titleformat{\section}{\Large\bfseries}{Appendix \thesection:}{0.5em}{}

\section{Proof of Section \ref{section:RUC-regularity}}
\label{appdx:proposition:growth}

\subsection{Proof of Remark \ref{remark:growth}}
\label{appdx:remark:growth}
\begin{proof}
For this remark, we only need to verify the uniform continuity for each $\log(f)$ and then apply Proposition \ref{proposition:growth} to finish the proof.  First, for case (1) where $f(x)\equiv L$, the Lipschitz continuity is obvious. For case (2) where 
$f(x)=\log(c+\|x\|)$ with $c>1$ and define
$\phi(t):=\log\log(c+t)$ for $t\ge 0$. Then $
\phi'(t)=\frac{1}{(c+t)\log(c+t)}\le \frac{1}{c\log c}$ for all $t\ge 0,$
and $\phi$ is Lipschitz continuous on $[0,\infty)$ with constant $1/(c\log c)$. Together with the 1-Lipschitz continuity of the Euclidean norm function $\|\cdot\|$, we obtain the $1/(c\log c)$-Lipschitz continuity of the composite function 
$\log(f(x)) = \phi(\|x\|)$.

Next, let us consider the case (3) where $f(x)=L+\|x\|^\nu$ with $L>0$ and $\nu>0$. Similar to case (2), we only need to prove the Lipschitz (or H\"older) continuity of  
$\phi(t):=\log(L+t^\nu)$ on $\mathbb{R}_{+}$.

When $0<\nu\le 1$, it holds for all $t,s\ge 0$ that $|t^\nu-s^\nu|\le |t-s|^\nu$. Together with the fact that $|\log a-\log b|\le |a-b|/\min\{a,b\}$ for all $a,b>0$, we obtain
\[
|\phi(t)-\phi(s)|
=\big|\log(L+t^\nu)-\log(L+s^\nu)\big|
\le \frac{|t^\nu-s^\nu|}{L}
\le \frac{1}{L}|t-s|^\nu.
\]
Therefore $\phi$ is $1/L$-H\"older continuous with exponent $\nu$. Again, combined with the 1-Lipschitz continuity of Euclidean norm function, we know $\log(f) = \phi(\|x\|)$ is also H\"older continuous when $\nu\leq1$.

When $\nu>1$, one can directly compute that 
\[
\sup_{t\geq0}\psi'(t)=\sup_{t\geq0}\frac{\nu t^{\nu-1}}{L+t^\nu} = L^{-\frac{1}{\nu}}(\nu-1)^{1-\frac{1}{\nu}} < +\infty.
\]
That is, $\phi$ is globally Lipschitz continuous. Similar to previous argument, we know $\log(f(x)) = \phi(\|x\|)$ is Lipschitz continuous when $\nu>1$. 

Finally, for case (4) where $f(x)=e^{\|x\|/r}$, we have  $\log(f(x))=\|x\|/r$, which is a standard globally $1/r$-Lipschitz continuous. This completes the proof of the remark.
\end{proof}

\begin{remark}
For completeness and application on the numerical experiment, we also provide a direct argument that quantifies the
RUC bound for case (3) with \( F(\cX) := \max\{L + \|x\|^\nu : x \in \cX\} \) with \( L > 0 \) and \( \nu > 0 \) when $\|x\|\ge1$. Let \( \epsilon > 0 \) and suppose \( d_{\rm H}(\cX, \cY) \leq \delta \). Let \( y^* \in \cY \) be such that \( F(\cY) = L + \|y^*\|^\nu \), and let \( x \in \cX \) satisfy \( \|x - y^*\| \leq \delta \). Then: \[ \left| \|x\| - \|y^*\| \right| \leq \|x - y^*\| \leq \delta. \] Assume without loss of generality \( F(\cX) \leq F(\cY) \), then: \begin{equation} \label{prop:power-1} \max\left\{ \frac{F(\cX)}{F(\cY)}, \frac{F(\cY)}{F(\cX)} \right\} = \frac{F(\cY)}{F(\cX)}\leq \frac{L+\|y^*\|^\nu}{L+\|x\|^\nu}\leq\frac{L+(\|x\|+\delta)^\nu}{L+\|x\|^\nu}. \end{equation} \textbf{(i)} \( 0 < \nu < 1 \). Using the subadditivity of \( r \mapsto r^\nu \), \( (\|x\| + \delta)^\nu \leq \|x\|^\nu + \nu \delta \|x\|^{\nu - 1}. \)\\Therefore, \[ \frac{F(\cY)}{F(\cX)} \leq 1 + \frac{\nu \delta \|x\|^{\nu - 1}}{L + \|x\|^\nu}. \] If \( \|x\| \geq 1 \), then \(\frac{F(\cY)}{F(\cX)} \leq \left(1 + \frac{\delta}{\|x\|}\right)^\nu, \, \text{so} \,\, d_{F}^{\rm rel}(\cX,\cY) \leq \left(1 + \frac{\delta}{\|x\|} \right)^\nu - 1. \) Choosing \( \delta = \|x\|\left((\epsilon + 1)^{1/\nu} - 1\right) \) suffices.\\ \textbf{(ii):} \( \nu \geq 1 \). By the mean value theorem: \( (\|x\| + \delta)^\nu - \|x\|^\nu \leq \nu \delta (\|x\| + \delta)^{\nu - 1}, \) so \[ \frac{F(\cY)}{F(\cX)} \leq 1 + \frac{\nu \delta (\|x\| + \delta)^{\nu - 1}}{L + \|x\|^\nu}. \] If \( \|x\| \geq 1 \), the argument reduces to the same inequality as in the case \( \nu < 1 \), and we can again choose: \[ \delta = \|x\|\left((\epsilon + 1)^{1/\nu} - 1\right). \] $\delta$ can also achieve large values based on $\|x\|$. In all cases, the RUC condition is satisfied, so \( F \) is RUC-regular.
\end{remark}

\subsection{Proof of Example \ref{example:gen-smooth}}
\label{appdx:example:gen-smooth}
\begin{proof}
    First, we want to prove the gradient‐growth bound under generalized smoothness. 
Set \(d = x - \bar x\) and define the scalar function
$g(t) \;=\; \bigl\|\nabla f(\bar x + t\,d)\bigr\|,t\in[0,1].$
By direct computation, we have 
\[
g'(t)
 \le \|\nabla^2 f(\bar x + t\,d)\|\;\|d\|.
\]
By the $\alpha$-generalized smooth assumption, we obtain
\[
g'(t)
\;\le\;\bigl(L_0 + L_1\,g(t)^\alpha\bigr)\,\|d\|
\;=\;L_0\,\|d\| \;+\; L_1\,\|d\|\,g(t)^\alpha.
\]
Integrating from \(0\) to \(t\) yields
\[
g(t)
\;\le\;g(0) \;+\; L_0\,\|d\|\,t
 \;+\; L_1\,\|d\|\int_0^t g(s)^\alpha\,ds.
\]
Because $t\in[0,1]$, we can upper bound the linear term $L_0\|d\|t$ by a constant $L_0\|d\|$. Then we define
\[
h(t) \;=\; g(0) + L_0\,\|d\|
  + L_1\,\|d\|\int_0^t g(s)^\alpha\,ds,
\]
so that \(g(t)\le h(t)\) for all \(t\in[0,1]\). Compute the gradient of $h$, we obtain
\[
h'(t) \;=\; L_1\,\|d\|\,g(t)^\alpha
\;\le\;L_1\,\|d\|\,h(t)^\alpha.
\]
Consider \(u(t)=h(t)^{1-\alpha}\).  Then $
u'(t)
=(1-\alpha)h'(t)h(t)^{-\alpha}
\le(1-\alpha)L_1\|d\|.
$
Consequently, 
\[
h(1)^{1-\alpha}
\;=\;u(1)
\;\le\;u(0)+(1-\alpha)\,L_1\,\|d\|
=\bigl[g(0)+L_0\|d\|\bigr]^{1-\alpha} + (1-\alpha)L_1\|d\|.
\]
With \(\|\nabla f(x)\| = g(1)\leq h(1)\) and \(\|d\|=\|x-\bar x\|\), for any reference point \(\bar x\) and all \(x\in\RR^d\), we have 
\[
\|\nabla f(x)\|
\;\le\;
\Bigr(\Bigl[\|\nabla f(\bar x)\| + L_0\,\|x-\bar x\|\Bigr]^{1-\alpha}
\;+\;(1-\alpha)\,L_1\,\|x-\bar x\|
\;\Bigr)^{\tfrac1{1-\alpha}}.
\]
Substituting to the Hessian bound of $\alpha$-generalized smoothness assumption, we can get
\begin{align*}
&\|\nabla^2 f(x)\|\;\le\;L_0 \;+\;L_1\,\|\nabla f(x)\|^\alpha \le 
L_0 \;+\;L_1\,\left(\Bigr(\Bigl[\|\nabla f(\bar x)\| + L_0\,\|x-\bar x\|\Bigr]^{1-\alpha}
\;+\;(1-\alpha)\,L_1\,\|x-\bar x\|
\;\Bigr)^{\tfrac1{1-\alpha}}\right)^\alpha\\
&\leq L_0 \;+\;L_1\,\Bigg(2^\frac{\alpha}{1-\alpha}\|\nabla f(\bar x)\| + 2^\frac{\alpha}{1-\alpha} L_0\,\|x-\bar x\|
\;+\;2^\frac{\alpha}{1-\alpha} \Bigr((1-\alpha)\,L_1\,\|x-\bar x\|
\;\Bigr)^{\tfrac1{1-\alpha}}\Bigg)^\alpha\\
& \leq L_0 \;+\;L_1\,\left[2^\frac{\alpha}{1-\alpha}\|\nabla f(\bar x)\| + 2^\frac{\alpha}{1-\alpha}L_0\,\|x-\bar x\|\right]^\alpha
\;+L_1 2^\frac{\alpha^2}{1-\alpha}\,\;\Biggr((1-\alpha)\,L_1\,\|x-\bar x\|
\;\Biggr)^{\frac{\alpha^2}{1-\alpha}}\\
&\leq L_0 \;+\;L_1\,2^\frac{\alpha}{1-\alpha}\|\nabla f(\bar x)\| + L_1\,2^\frac{\alpha^2}{1-\alpha}L_0\,\|x-\bar x\|^\alpha
\;+L_1 2^\frac{\alpha^2}{1-\alpha}\,\;\Biggr((1-\alpha)\,L_1\,\|x-\bar x\|
\;\Biggr)^{\frac{\alpha}{1-\alpha}}\\
&\overset{(i)}{\leq} 
L_0 \;+\;L_1\,2^\frac{\alpha}{1-\alpha}\|\nabla f(\bar x)\| + L_1\,2^\frac{\alpha^2}{1-\alpha}L_0+L_1\,2^\frac{\alpha^2}{1-\alpha}L_0\,\|x-\bar x\|^{\frac{\alpha}{1-\alpha}}
\;+L_1 2^\frac{\alpha^2}{1-\alpha}\,\;\Biggr((1-\alpha)\,L_1\,\|x-\bar x\|
\;\Biggr)^{\frac{\alpha}{1-\alpha}}\\
& \leq A_0(\bar{x})+A_1(\bar{x})\|x-\bar{x}\|^{\frac{\alpha}{1-\alpha}},
\end{align*}
where $(i)$ is from that $\|x-\bar x\|^\alpha \leq 1+\|x-\bar x\|^{\frac{\alpha}{1-\alpha}}$
and 
\[
\begin{aligned}
A_0(\bar{x}) &:= L_0 
+ L_1\,2^{\tfrac{\alpha}{1-\alpha}}\,\|\nabla f(\bar{x})\| 
+ L_1\,2^{\tfrac{\alpha^2}{1-\alpha}}\,L_0, \\
A_1(\bar{x}) &:= L_1\,2^{\tfrac{\alpha^2}{1-\alpha}}\,L_0 
+ L_1\,2^{\tfrac{\alpha^2}{1-\alpha}}\,\left((1-\alpha)\,L_1\right)^{\tfrac{\alpha}{1-\alpha}}.
\end{aligned}
\]
This completes the proof.
\end{proof}

\section{Proof of Section \ref{section:cvg-analysis}}
\subsection{Proof of Lemma \ref{lemma:stag-descent}}
\label{appdx:lemma:stag-descent}
\begin{proof}     Note that by Assumption \ref{assumption:RMSS} and the definition of $\bL(\cdot)$, the standard argument leads to  
$$
f(\bbxte)\leq f(\bbxt) + \langle \nabla f(\bbxt), (\bbxte)^\top-(\bbxt)^\top\rangle + \frac{\bL([\bbxt,\bbxte])}{2}\|\bar{\mathbf{x}}^{t}-\bar{\mathbf{x}}^{t+1}\|^2.
$$
By Remark \ref{remark:Stag-point-membership}, we know the line segment $[\bbxt,\bbxte]\subseteq\cX_i^{l(i,t)}$ for any $i\in[m]$. Consequently, we know $\bL([\bbxt,\bbxte])\leq L_{\min}^t:=\min_i L_i^t$. Using the fact that $\bbxte = \bbxt - (\bone/m)^\top\ctdb(\byt)$, we obtain
\begin{eqnarray}
\label{lm:stag-descent-1}
f(\bbxte) &\leq& f(\bbxt) + \frac{1}{m}\left\langle \bone\nabla f(\bbxt)^\top - \byt + \byt, -\ctdb(\byt)\right\rangle + \frac{L_{\min}^t}{2}\left\|(\bone/m)^\top\ctdb(\byt)\right\|^2 \\
& \leq & f(\bbxt)  -  \frac{(2\theta-1)L_{\min}^t}{2m} \left\|\ctdb(\byt)\right\|_F^2 + \frac{1}{m}\left\langle \bone\nabla f(\bbxt)^\top - \byt , -\ctdb(\byt)\right\rangle  \nonumber
\end{eqnarray} 
where the last inequality is due to the fact that $\beta_i^t = \theta L_i^t$, and $\langle\mathcal{T}_\beta^\delta(v),v\rangle \geq \beta\|\mathcal{T}_\beta^\delta(v)\|^2$ for any $\beta,\delta>0$ and any vector $v\in\RR^d$.  In the above inequality, the last term can be decomposed as 
\begin{eqnarray} \label{lm:stag-descent-2}
&&\frac{1}{m}\left\langle\bone\nabla f(\bbxt)^\top - \byt , -\ctdb(\byt)\right\rangle \\
&=& \underbrace{\frac{1}{m}\left\langle \bone\nabla f(\bbxt)^\top - \bone(\omega^t)^\top , -\ctdb(\byt)\right\rangle}_{T_1} 
+ \underbrace{\frac{1}{m}\left\langle\bone(\omega^t)^\top - \bone\bbyt , -\ctdb(\byt)\right\rangle}_{T_2} + \underbrace{\frac{1}{m}\left\langle \bone\bbyt - \byt , -\ctdb(\byt)\right\rangle}_{T_3}. \nonumber
\end{eqnarray}  
For the term $T_1$, we have 
\begin{align*}
   T_1
   & = \frac{1}{m}\sum_{i=1}^m\left \langle \nabla f_i(\bar{\mathbf{x}}^{t})^\top- \nabla f_i(x_i^t)^\top, - (\bone/m)^\top\ctdb(\byt)\right\rangle \\
   & \leq\frac{1}{m}\sum_{i=1}^m\frac{1}{2L_{\min}^t}\left\|\nabla f_i(\bar{\mathbf{x}}^{t})- \nabla f_i(x_i^t)\right\|^2+ \frac{L_{\min}^t}{2} \left\|(\bone/m)^\top\ctdb(\byt)\right\|^2 \\
   &\leq \sum_{i=1}^m\frac{\bL([\xit,(\bbxt)^\top])^2}{2mL_{\min}^t}\|(\bar{\mathbf{x}}^t)^\top-x_i^t\|^2+ \frac{L_{\min}^t}{2m}\left\|\ctdb(\byt)\right\|_F^2\\
   &\overset{(i)}{\leq}\frac{L_{\text{min}}^t}{2m}\|\mathbf{1}\bar{\mathbf{x}}^t-\mathbf{x}^t\|^2_F+\frac{L^t_{\min}}{2m}\left\|\ctdb(\byt)\right\|^2_F,
\end{align*} 
where (i) is due to Remark \ref{remark:Stag-point-membership}, which states that the line segment $[\xit,(\bbxt)^\top]\subseteq\cX_j^{l(j,t)}$ for any $j\in[m]$. Thus $\bL([\xit,(\bbxt)^\top])\leq L^t_{\min}, \forall i\in[m]$. 
Next,  we bound the term $T_2$ by
\begin{eqnarray}
    T_2  &=&  \left \langle \omega^t - (\bbyt)^\top, - (\bone/m)^\top\ctdb(\byt)\right\rangle\nonumber\\
    &\leq& \frac{1}{2 L_{\min}^t}\left\|\omega^t-(\bbyt)^\top\right\|^2 + \frac{L_{\min}^t}{2m}\left\|\ctdb(\byt)\right\|^2_F. \nonumber
\end{eqnarray}
Finally, we bound the term $T_3$ by 
\begin{align*}
    T_3 \leq \frac{1}{2mL_{\min}^t}\left\|\bone\bbyt - \byt\right\|^2_F + \frac{L_{\min}^t}{2m}\left\|\ctdb(\byt)\right\|_F^2.
\end{align*}
Combining \eqref{lm:stag-descent-1}, \eqref{lm:stag-descent-2}, and the bounds on $T_1,T_2,T_3$, we prove \eqref{lm:stag-descent-main}.
\end{proof}

\subsection{Proof of Lemma \ref{lemma:stag-descent-gradmap}}
\label{appdx:lemma:stag-descent-gradmap}

\begin{proof}
First, by Young's inequality: 
\begin{align}
\label{lm:GrdMap-error-1}
&\left\|\ctdb\big(\bone\nabla f(\bbxt)^\top\big)\right\|_F^2 \,\, \leq  \,\,4\left\|\ctdb\big(\byt\big)\right\|_F^2 + \frac{4}{3}\left\|\ctdb\big(\bone\nabla f(\bbxt)^\top\big) - \ctdb\big(\byt\big)\right\|_F^2\nonumber\\
& \quad\overset{(i)}{\leq} \,\, 4\left\|\ctdb\big(\byt\big)\right\|_F^2 + \frac{4}{3(\theta L_{\min}^t)^2} \left\|\bone\nabla f(\bbxt)^\top- \byt\right\|_F^2\\
& \quad\leq \,\, 4\left\|\ctdb\big(\byt\big)\right\|_F^2 + \frac{4}{(\theta L_{\min}^t)^2} \Big(\left\|\bone\nabla f(\bbxt)^\top-\bone(\omega^t)^\top\right\|_F^2 + \left\|  \bone(\omega^t)^\top - \bone\bbyt\right\|_F^2 + \left\|\bone\bbyt - \byt\right\|_F^2\Big)\nonumber\\
& \quad\!\overset{(ii)}{\leq}  \, 4\left\|\ctdb\big(\byt\big)\right\|_F^2 + \frac{4}{(\theta L_{\min}^t)^2} \Big((L_{\min}^t)^2\big\|\bone\bbxt-\bxt\big\|_F^2 + m\left\|   (\omega^t)^\top -  \bbyt\right\|_F^2 + \left\|\bone\bbyt - \byt\right\|_F^2\Big).\nonumber
\end{align} 
In (i), for each row $i$, $\cT_{\beta_i^t}^\delta(v)$ is projecting the vector $v/\beta_i^t$ to the ball $B(0,\delta)$. Then one can utilize the non-expansive property of projection operator to obtain the bound. And (ii) is because
\begin{align*}
& \left\|\bone\nabla f(\bbxt)^{\top}-\bone(\omega^t)^{\top}\right\|_F^2 = \frac{1}{m}\Big\|\sum_{i=1}^m\nabla f_i(\bbxt)-\nabla f_i(\xit)\Big\|^2\\    
&\qquad\qquad\qquad\quad\leq  \sum_{i=1}^m\!\big\| \nabla f_i(\bbxt)-\nabla f_i(\xit)\big\|^2 
\leq  (L_{\min}^t)^2\big\|\bone\bbxt-\bxt\big\|_F^2,
\end{align*} 
where the use of $L_{\min}^t$ follows the same logic of bounding $T_1$ in the proof of Lemma \ref{lemma:stag-descent}, see Appendix \ref{appdx:lemma:stag-descent}. Denote $P_\delta(\cdot)$ the projection to a $\delta$-radius ball, then 
$$\gamma\big\|\ctdbi(\nabla f(\bbxt))\big\|\overset{(i)}{\geq} \frac{\beta_i^t}{\beta_{\min}^t}\big\|\ctdbi(\nabla f(\bbxt))\big\| = \frac{\beta_i^t}{\beta_{\min}^t}\big\|P_\delta(\nabla f(\bbxt)/\beta_i^t)\big\|^2\overset{(ii)}{\geq} \big\|\cT_{\beta_{\min}^t}^\delta(\nabla f(\bbxt))\big\|,$$
where $\gamma = \left\{2,\frac{1+\eta}{2\eta}\right\}$ is defined in \eqref{eqn:ruc-r0} for RUC property, (i) is by Proposition \ref{proposition:stager-dist-bound}, and (ii) is because $P_\delta$ is a projection and the fact that $\beta_i^t/\beta_{\min}^t\geq1$. Substituting this bound to \eqref{lm:GrdMap-error-1} proves that 
\begin{align}
\label{lm:GrdMap-error-main}
\frac{m}{\gamma^2}\big\|\cT_{\beta_{\min}^t}^\delta(\nabla f(\bbxt))\big\|^2 \leq & \,\,\, 4\left\|\ctdb\big(\byt\big)\right\|_{F}^2 + \frac{4(L_{\min}^t)^2\big\|\bone\bbxt-\bxt\big\|_{F}^2}{(\theta L_{\min}^t)^2} \\
& + \frac{4}{(\theta L_{\min}^t)^2} \Big(m\left\|(\omega^t)^{\top} -  \bbyt\right\|_{F}^2 + \left\|\bone\bbyt - \byt\right\|_{F}^2\Big).  \nonumber
\end{align}  
Based on Lemma \ref{lemma:stag-descent}, we can take the weighted sum \eqref{lm:stag-descent-main}$+$\eqref{lm:GrdMap-error-main}$\times\frac{(\theta-2)L_{\min}^t}{8m}$ to obtain
\begin{eqnarray}
\label{lm:stag-descent-5}
f(\bbxte) &\leq& f(\bbxt)  -  \frac{(\theta-2)L_{\min}^t}{8\gamma^2} \left\|\ctdb(\nabla f(\bbxt))\right\|_F^2  -  \frac{(\theta-2)L_{\min}^t}{2m}\left\|\ctdb(\byt)\right\|_F^2 \\
&& +  \Big(\frac{L_{\text{min}}^t}{2m} + \frac{(\theta-2)L_{\min}^t}{2m\theta^2}\Big)\|\mathbf{1}\bar{\mathbf{x}}^t-\mathbf{x}^t\|^2_F  
 + \Big(\frac{1}{2mL_{\min}^t} + \frac{\theta-2}{2m\theta^2L_{\min}^t}\Big)\|\bone\bbyt - \byt\|^2_F\nonumber \\
 & &+ \Big(\frac{1}{2 L_{\min}^t}+\frac{\theta-2}{2\theta^2 L_{\min}^t}\Big)\|\omega^t-(\bbyt)^\top\|^2   \nonumber. 
\end{eqnarray} 
Then using the fact that $\max\left\{\frac{1}{2}+\frac{\theta-2}{2\theta^2}:\theta>2\right\} = 9/16$, we prove the lemma. 
\end{proof}

\subsection{Proof of Lemma \ref{lemma: potential descent 2}}
\label{appdx:lemma:potential}
\begin{proof}
The proof of this lemma consists of several steps. Let us go through them step by step. \vspace{0.2cm}

\noindent{\bf Step 1: Establishing contractive consensus error bound.}  Note that for arbitrary sequence $\bvte = \bW\bvt + {\bf u}^t$ with the matrix $\bW$ satisfying Assumption \ref{assumption:mix-spectral}, it is a standard result in distributed optimization that 
\begin{eqnarray}
\label{eqn:mixing}
\big\|\bvte-\bone\bbvte\big\|_F^2 &=& \big\|\bW\bvt+{\bf u}^t - \bone(\bbvt+\bar{{\bf u}}^t)\big\|_F^2\\
& \leq & (1+\zeta)\big\|\bW\bvt\! - \bone \bbvt\big\|_F^2 + \Big(1\!+\!\frac{1}{\zeta}\Big)\big\|{\bf u}^t \!- \bone\bar{{\bf u}}^t\big\|_F^2\nonumber\\
& \leq & (2-\eta)\eta^2\big\|\bvt - \bone \bbvt\big\|_F^2 + \frac{2-\eta}{1-\eta}\big\|{\bf u}^t - \bone\bar{{\bf u}}^t\big\|_F^2\nonumber\\
& \leq & \eta\big\|\bvt\! - \bone \bbvt\big\|_F^2 + \frac{2}{1-\eta}\big\|{\bf u}^t \big\|_F^2\nonumber 
\end{eqnarray} 
where we select $\zeta = 1-\eta$ and use the fact that $\|{\bf u}^t\! - \bone\bar{{\bf u}}^t\|_F^2\leq \|{\bf u}^t\|_F^2$ for arbitrary ${\bf u}^t$. Applying this inequality to the $\bxt$ sequence, using the Proposition \ref{proposition:stager-dist-bound}, and then taking expectation gives 
\begin{equation}
\label{lm:potential-1}
\EE\big[L_{\min}^{t+1}\|\bxte-\bone\bbxte\|^2_F\big]\leq \EE\Big[\eta\gamma L_{\min}^t\big\|\bxt\! - \bone \bbxt\big\|_F^2 + \frac{2\gamma L_{\min}^t}{1-\eta}\big\|\ctdb(\byt) \big\|_F^2\Big].
\end{equation}
Similarly, applying this inequality to the $\byt$ sequence and using Proposition \ref{proposition:stager-dist-bound} gives 
\begin{equation}
\label{lm:potential-2}
\frac{\|\byte-\bone\bbyte\|^2_F}{L_{\min}^{t+1}} \leq \frac{\eta\gamma\|\byt-\bone\bbyt\|_F^2}{L_{\min}^{t}} + \frac{2\gamma\|\bGte- \bGt\|_F^2}{(1-\eta)L_{\min}^t}.
\end{equation}
Next, let us bound the last term by bounding each row of the difference $\bGte- \bGt$. In case agent $i$ does not start a new epoch in the next iteration $t+1$, then by Line 10 of \texttt{CGTVR-stag}, we have
\begin{align*}
    \frac{\|{g}_i^{t+1}\!-g_i^t\|^2}{L_{\min}^t} &= \frac{\|\frac{1}{B_i}\!\!\sum_{j\in\cB_i^t}\!(\nabla\fij(\xite)-\nabla\fij(\xit)\!)\|^2}{L_{\min}^t} \\
    &\leq \frac{\bL([\xit,\xite])^2}{L_{\min}^t}\|x_i^{t+1}\!-x_i^{t}\|^2\leq L_{\min}^t\|x_i^{t+1}\!-x_i^{t}\|^2 
\end{align*}
where the replacement of $\bL([\xit,\xite])$ by $L_{\min}^t$ is due to Remark \ref{remark:Stag-point-membership}. In case agent $i$ does start a new epoch in the next time step $t+1$, then $g_i^{t+1} = \nabla f_i(x_i^{t+1})$ due to Line 6 of \texttt{CGTVR-stag}. Consequently, the following inequality holds that 
\begin{eqnarray}
\frac{\|\gite-\git\|^2}{L_{\min}^t}  &\leq&  \frac{2\|\nabla f_i(\xite)-\nabla f_i(\xit)\|^2}{L_{\min}^t} 
+ \frac{2\|\nabla f_i(\xit) - \git\|^2}{L_{\min}^t}\nonumber \\
&\leq & 2L_{\min}^t\|\xite\!-\xit\|^2  
+ \frac{2\|\nabla f_i(\xit) - \git\|^2}{L_{\min}^t}.\nonumber
\end{eqnarray}  
Comparing the upper bounds in two cases, find that the above bound holds no matter the agent $i$ starts a new epoch at time $t+1$ or not. Therefore, we can sum up the above bound over $i$ and obtain
\begin{equation}
\label{lm:potential-3}
\frac{2\gamma\|\bGte- \bGt\|_F^2}{(1-\eta)L_{\min}^t} \leq \frac{4\gamma L_{\min}^t}{1-\eta}\|\bxte\!-\bxt\|^2_F  
+ \sum_{i=1}^m\frac{4\gamma\|\nabla f_i(\xit) - \git\|^2}{(1-\eta)L_{\min}^t},
\end{equation}
where the $i$-th term in the last summation error term can be bounded in expectation by Lemma \ref{lemma:bound variance} because it can be viewed as a special case of this lemma with only $m=1$ agent. Therefore, combining \eqref{lm:potential-2} and \eqref{lm:potential-3}, taking expectation on both sides, and then utilizing the bound in Lemma \ref{lemma:bound variance} gives 
\begin{align}
\label{lm:potential-4}
\EE\bigg[\frac{\|\byte - \bone\bbyte\|^2_F}{L_{\min}^{t+1}}\bigg]   \leq & \,\,\,\EE\bigg[\frac{\eta\gamma\|\byt -\bone\bbyt\|_F^2}{L_{\min}^{t}} + \frac{4\gamma L_{\min}^t}{1-\eta}\|\bxte-\bxt\|^2_F\bigg]\\
& + \EE\bigg[\frac{4\gamma^2}{1-\eta}\sum_{i=1}^m \sum_{r=l(i,t)}^{t-1}\frac{L_{\min}^{r}}{B_i}\left\|x_i^{r+1}-x_i^{r}\right\|^2\bigg].\nonumber
\end{align} 
Hence we finish the first step of establishing contractive consensus error bound. \vspace{0.25cm}

\noindent \textbf{Step 2: Telescoping potential function difference.} By definition of $P_t$,  direct computation gives
\begin{eqnarray}
P_{t+1}-P_t & = & \EE\Big[f(\bbxte)-f(\bbxt)\Big] +   \alpha_2\EE\bigg[\frac{\|\byte-\bone\bbyte\|_F^2}{L_{\min}^{t+1}}- \frac{\|\byt-\bone\bbyt\|_F^2}{L_{\min}^{t}}\bigg] \nonumber\\
& & + \alpha_1\EE\Big[L_{\min}^{t+1}\|\bxte-\bone\bbxte\|_F^2- L_{\min}^{t}\|\bxt-\bone\bbxt\|_F^2\Big] \nonumber.
\end{eqnarray}
In the above inequality, let us bound the first term by Lemma \ref{lemma:stag-descent-gradmap} and Lemma \ref{lemma:bound variance}, bound the second term by \eqref{lm:potential-4}, bound the third term by \eqref{lm:potential-1}, 
we obtain 
\begin{eqnarray}
\label{lm:potential-5}
P_{t+1}-P_t & \leq & - \EE\bigg[\frac{(\theta-2)L_{\min}^t}{8\gamma^2}\big\|\cT_{\beta_{\min}^t}^\delta(\nabla f(\bbxt)\big\|^2 + \Big(\frac{\theta-2}{2m}-\frac{2\alpha_1\gamma}{1-\eta}\Big)L_{\min}^t\big\|\ctdb(\byt)\big\|_F^2\bigg] \\
& & -\EE\bigg[\Big((1-\eta\gamma)\alpha_1-\frac{9}{16m}\Big)L_{\min}^t\|\bxt-\bone\bbxt\|_F^2 + \Big((1-\eta\gamma)\alpha_2-\frac{9}{16m}\Big)\frac{\|\byt-\bone\bbyt\|_F^2}{L_{\min}^t}\bigg]\nonumber\\
& & + \EE\bigg[\frac{4\alpha_2\gamma L_{\min}^t}{1-\eta}\|\bxte-\bxt\|_F^2 + \Big(\frac{9\gamma}{16m}+\frac{4\alpha_2\gamma^2}{1-\eta}\Big)\sum_{i=1}^m \!\sum_{r=l(i,t)}^{t-1}\!\!\frac{L_{\min}^{r}}{B_i}\!\left\|x_i^{r+1}-x_i^{r}\right\|^2\bigg] \nonumber.
\end{eqnarray} 
Note that $\gamma = \min\big\{2,\frac{1+\eta}{2\eta}\big\}\leq 2$ and $1-\gamma\eta\geq \frac{1-\eta}{2}>0$, taking telescopic sum of \eqref{lm:potential-5}, we obtain  
\begin{align}
\label{lm:potential-6}
& P_{T}\!-\!P_0  \!\leq\!  -\! \sum_{t=0}^{T-1}\EE\bigg[\frac{(\theta\!-\!2)L_{\min}^t}{32}\big\|\cT_{\beta_{\min}^t}^\delta(\nabla f(\bbxt))\big\|^2 \!+\! \Big(\frac{\theta\!-\!14}{2m}\!-\!\frac{4\alpha_1}{1-\eta}\!-\!\frac{120\alpha_2}{1-\eta}\Big)L_{\min}^t\big\|\ctdb(\byt)\big\|_F^2\bigg] \\
& \qquad\qquad\,\, -\EE\bigg[\Big(\frac{\alpha_1(1-\eta)}{2} - \frac{25}{4m} - \frac{120\alpha_2}{1-\eta}\Big)L_{\min}^t\|\bxt-\bone\bbxt\|_F^2 + \Big(\frac{(1-\eta)\alpha_2}{2}-\frac{9}{16m}\Big)\frac{\|\byt-\bone\bbyt\|_F^2}{L_{\min}^t}\bigg].\nonumber
\end{align} 
Therefore, if we select the parameters so that 
$$\alpha_1 = \frac{14}{(1-\eta)m} + \frac{600}{(1-\eta)^3m},\quad\qquad\alpha_2 = \frac{5}{2(1-\eta)m},\quad\qquad\theta \geq \theta_0:= 28+\frac{1424}{(1-\eta)^2}+\frac{9600}{(1-\eta)^4},$$
then the coefficients in inequality \eqref{lm:potential-6} satisfies
$$\frac{(1-\eta)\alpha_2}{2}-\frac{9}{16m} = \frac{1}{16m}, \qquad \qquad  \frac{(1-\eta)\alpha_1}{2} - \frac{25}{4m} - \frac{120\alpha_2}{1-\eta} = \frac{3}{4m},$$
$$\frac{\theta-14}{2m}-\frac{4\alpha_1}{1-\eta}-\frac{120\alpha_2}{1-\eta} \geq \frac{\theta}{4m},$$
which proves the lemma. 
\end{proof}

\bibliographystyle{plain}
\bibliography{ref}
\end{document}